\documentclass[10pt,letterpaper,table,xcdraw]{amsart}
\usepackage[foot]{amsaddr}
\usepackage{graphicx}
\usepackage[margin=1in]{geometry}
\usepackage{amssymb,dsfont,amsfonts}
\usepackage{mathtools}
\usepackage[utf8]{inputenc}
\usepackage{cmap}
\usepackage[english]{babel}
\usepackage[pdfdisplaydoctitle = true,
colorlinks = true,
urlcolor = blue,
citecolor = blue,
linkcolor = blue,
pdfstartview = FitH,
pdfpagemode = UseNone,
bookmarksnumbered = true,
unicode=true]{hyperref}
\usepackage{algpseudocode}
\usepackage{placeins}
\usepackage{array}
\usepackage{booktabs}
\usepackage{multirow}
\usepackage{mathrsfs}
\usepackage{tikz}
\usepackage{bbm}

\usepackage{cite}

\usepackage{caption}
\usepackage{pgfplots}
\pgfplotsset{compat=1.18}
\usetikzlibrary{patterns,plotmarks}
\usepgfplotslibrary{groupplots,dateplot}
\usetikzlibrary{shapes.arrows,calc,math}

\usepackage{subcaption}
\usepackage{enumitem}
\setlist{topsep=2pt,itemsep=0pt,parsep=0pt,partopsep=0pt}

\usepackage{amsmath}

\usepackage{algorithm}
\usepackage{bm}

\usepackage{etoolbox}
\AtBeginEnvironment{thebibliography}{\scriptsize\setlength{\itemsep}{0pt}\setlength{\parsep}{0pt}\setlength{\parskip}{0pt}}
\def\RR{\mathbb{R}}
\def\NN{\mathbb{N}}
\def\HH{\mathbf{H1}}
\def\HHCS{\mathbf{H1CS}}
\def\hh{\mathbf{H2}}
\def\hhi{\mathbf{iH2}}
\def\hhcs{\mathbf{H2CS}}
\def\MM{{\bf M}}
\def\x{\bm{x}}
\def\vb{\bm{v}}

\def\p{\bm{p}}
\def\y{\bm{y}}

\def\q{\bm{q}}

\def\balpha{\bm{\alpha}}
\def\bbeta{\bm{\beta}}
\def\bgamma{\bm{\gamma}}

\DeclareMathOperator{\supp}{supp}

\DeclareMathOperator{\cs}{cs}

\renewcommand{\arraystretch}{1.28}

\def\vv{\bm{v}}

\newtheorem{thm}{Theorem}[section]
\newtheorem{prop}{Proposition}[section]

\newtheorem{h1prop}[thm]{Proposition}
\newtheorem{rem}[thm]{Remark}

\newtheorem{exmp}[thm]{Example}

\newcommand{\abs}[1]{\lvert#1\rvert}
\newcommand{\set}[1]{\left\{#1\right\}}

\title[Leveraging CD Kernels to Strengthen Moment-SOS Relaxations]{Leveraging Christoffel-Darboux Kernels to Strengthen Moment-SOS Relaxations}

\author{Sre{\'c}ko  {\DH}ura\v{s}inovi{\'c}$^{1,2}$}
\email{SRECKO001@e.ntu.edu.sg}

\author{Perla Azzi$^{1}$}
\email{perla.azzi@cnrsatcreate.sg}

\author{Jean-Bernard Lasserre$^{3,4}$}
\email{lasserre@laas.fr}

\author{Victor Magron$^{3,5}$}
\email{vmagron@laas.fr}

\author{Olga Mula$^{6}$}
\email{o.mula@tue.nl}

\author{Jun Zhao$^{2}$}
\email{junzhao@ntu.edu.sg}

\address[1]{CNRS@CREATE LTD, Singapore}
\address[2]{College of Computing and Data Science, Nanyang Technological University}
\address[3]{LAAS-CNRS}
\address[4]{Toulouse School of Economics}
\address[5]{Institute of Mathematics from Toulouse}
\address[6]{Eindhoven University of Technology, Department of Mathematics and Computer Science}

\begin{document}
\raggedbottom

\begin{abstract}
The classical Moment-SOS hierarchy allows one to approximate the global minimum of a polynomial optimization problem through semidefinite relaxations of increasing accuracy. For many optimization problems, however, moving to higher relaxation orders quickly becomes computationally prohibitive. While this difficulty is commonly addressed by exploiting structural properties such as symmetry or sparsity, we explore a complementary direction: strengthening a low-order relaxation using information contained in its own pseudo-moment solution. To this end, we leverage Christoffel-Darboux kernels to identify regions that are compatible with the current pseudo-moment information and use them to construct additional polynomial constraints. We propose two mechanisms: the first, denoted $\HH$, is based on successive global Christoffel restrictions, while the second, denoted $\hh$, combines marginal Christoffel information with an available local solution. For $\HH$, we establish a uniform separation property for successive pseudo-moment solutions and geometric convergence of the fixed-order bounds on a particular subclass of knapsack problems. Moreover, for specific Boolean quadratic problems, we derive an explicit $\hh$ selection rule and establish conditions under which the resulting marginal constraints preserve all global minimizers. Numerical experiments on dense and sparse nonconvex quadratically constrained quadratic programs show that the resulting strengthened low-order relaxations can significantly improve the initial bounds, sometimes attaining exactness, while remaining substantially cheaper than solving higher-order relaxations.
\end{abstract}
\maketitle
\thispagestyle{empty}

\section{Introduction}
\label{sec:intro}

\noindent
Polynomial optimization revolves around {minimizing or maximizing} a polynomial over a basic closed semialgebraic set, i.e., a set corresponding to the conjunction of finitely many polynomial equalities and inequalities. Generally, these problems are NP-hard \cite{laurent2009sums}.
However, in \cite{lasserre2001global}, the author provided a way of approximating the global optimum of a polynomial optimization problem (POP) by a hierarchy of semidefinite programming (SDP) relaxations. Finite convergence guarantees, under mild conditions, were later derived in \cite{FiniteCO-Nie2014}.
We {refer the reader to}
\cite{Lasserre2009Book,Henrion-book2020,Magron-book2023,Nie-book2023} and references therein for a detailed analysis of the hierarchy and its use in real-life applications.
\par
The main bottleneck of the proposed hierarchy, often called \textit{Moment-Sum of Squares} (Moment-SOS for short) or \textit{Lasserre's} hierarchy, is its scalability.
Indeed, this hierarchy involves semidefinite positive constraints on matrices of size up to $\binom{n+d}{d}$, where $n$ is the ambient space dimension and $d$ is the relaxation order, rendering it intractable if either of these parameters is large.
While it is true that $\binom{n+d}{d} = \mathcal{O}(n^d)$ for a fixed $n$, the curse of dimensionality becomes apparent when attempting to solve higher order relaxations with state-of-the-art SDP solvers.
Fortunately, many POP instances possess intrinsic structural properties whose adequate exploitation can lead to an increase in computational efficiency and mitigate the curse of dimensionality.\\
For example, more efficient relaxations that exploit \textit{symmetries} were introduced in \cite{symmetry-Reiner2013}. Furthermore, detecting and taking advantage of the problem's \textit{sparsity} constitutes another important direction in this domain, where some of the commonly studied sparsity patterns are \textit{correlative sparsity} \cite{waki2006sums,LasserreCSP2006}, \textit{term sparsity} \cite{TS-Wang2019, TSSOS-Wang2021} and \textit{ideal sparsity} \cite{Ideal-Korda2024}.
\par
Alternatively, one may try to identify some additional properties of the problem and translate them into new constraints. These new constraints, when coupled with the initial ones, often produce tighter relaxation bounds at each step of the hierarchy.
For instance, strengthening the first-order moment relaxations by adding redundant quadratic constraints has been studied in \cite{yang2024bnn}.
In \cite{Tight-Nie2019}, the author derives additional constraints involving \textit{Lagrange multipliers} associated to the original problem. These multipliers are represented as polynomials in the initial decision variables, and the proposed approach is demonstrated to provide better relaxation bounds whenever the optimal value is attained at a critical point of the associated Lagrangian. Combining correlative sparsity and Lagrange multiplier expressions is presented in \cite{Tight+CD-Qu2024}.
In \cite{ghaddar2016dynamic}, the authors propose a dynamic inequality generation scheme to construct valid polynomial inequalities.
Strengthened relaxation bounds using \textit{shift} and \textit{multiplication} operators, which arise when one tries to express optimality conditions of a given POP in a  positive semidefinite manner, are obtained in \cite{Strengthening-wang2024}.
One common drawback of the last two above-mentioned approaches is that they require to solve auxiliary relaxations at higher orders.
\par\smallskip
{\bf Contribution.} We introduce a complementary approach leveraging Christoffel-Darboux kernels.
Recent works, such as \cite{SortingOut-Pauwels2016,ChristoffelBook-Lasserre2022}, have demonstrated the broader applicability of these kernels, particularly in addressing data analysis challenges, paving the way for their use in optimization. In particular, results from \cite{Lasserre-CDK2024_NumALG}, reveal a very deep connection between Christoffel-Darboux kernels and SOS  polynomials, which are one of the main building blocks of the Moment-SOS hierarchy.
Our main contribution lies in exploiting these insights for strengthening classical hierarchy bounds, namely:
\begin{itemize}
    \item We argue that utilizing the Christoffel-Darboux kernel constructed at a fixed level $d$ of the hierarchy enables the construction of specific polynomial constraints that effectively restrict the feasible set. This approach yields significant improvements in bounds without incurring the computational cost of solving often intractable higher-order relaxations.
    Furthermore, we present and analyze the properties of two distinct heuristics for constructing these additional constraints;

\item For the iterative heuristic $\HH$, we establish uniform separation of successive pseudo-moment matrices, quantifying the effect of repeatedly restricting a fixed-order relaxation.
For a one-split $0$--$1$ knapsack family, we additionally
establish geometric convergence of the strengthened
first-order relaxation bounds under explicit assumptions;

\item For Boolean Ising/QUBO problems, we express the $\hh$ selection rule in terms of marginal pseudo-moments and relate it to the classical coefficient-dominance condition. We then establish sufficient conditions for the validity of the added constraints and show how these conditions can be applied iteratively;

    \item We demonstrate the effectiveness of our approach on a variety of quadratically constrained quadratic programs (QCQPs) across different dimensions.
\end{itemize}
\noindent

The rest of the paper is structured as follows: in Section \ref{sec:POP}, we recall some important properties of the main ingredients used in this paper: the Moment-SOS hierarchy and the Christoffel-Darboux kernels. In Section \ref{sec:Heuristics}, we introduce two different heuristic methods based on Christoffel-Darboux kernels, aimed at strengthening the classical relaxation bounds. Finally, in Section \ref{sec:num_exp}, we illustrate the performance of both heuristics on various classes of QCQPs and provide guidelines on how to use these heuristics in practice for more general problems.

\subsection{Notation} We use Roman letters to denote scalars, and boldfaced letters to represent vectors and matrices.
Let $\RR[\x]=\RR[x_1,\dotsc,x_n]$ be the ring of real $n$-variate polynomials, where $\bm{x}=(x_1,\dotsc,x_n)$.
For $d\in \NN$, we denote by $\RR_d[\x]$  the finite-dimensional space of polynomials of degree less than $d$, equipped with a standard monomial basis $\vb_d:=(1,x_1,\dots,x_n,x_1^2,\dots,x_n^d)$ of size $s(n,d):=\binom{n+d}{d}$. A positive semidefinite (PSD) matrix ${\bf M}$ is denoted by ${\bf M}\succeq 0$, and the set of symmetric (resp.~symmetric PSD) matrices of size $n$ is denoted  by $\mathcal{S}^n$ (resp. $\mathcal{S}^n_+$).

\noindent
Given $(d,n)\in \NN^2$, we define the multi-index set $\NN^n_d:=\set{\balpha=(\alpha_i,\dots,\alpha_n)\in \NN^n \mid \sum_{i=1}^n \alpha_i\leq d}$.
Consequently, given $\x\in\RR^n$ and  $f\in \RR_{d}[\x]$, we will often write $f(\x)=\sum_{\balpha\in \NN^n_{d}} f_{\balpha} \x^{\balpha}$ with $f_{\balpha}\in \RR$, $\x^{\balpha}=x_1^{\alpha_1}\dotsc x_n^{\alpha_n}$. The support of $f$ is $\supp(f):=\set{\balpha \in \NN^n ~\mid~ f_{\balpha} \neq 0}$.

\noindent
The set of real-valued sequences indexed by elements of a set $S$ is denoted by $\RR^S$. The set of finite signed Borel measures (resp. positive finite Borel measures) supported on a set $K$ is denoted by $\mathcal{M}(K)$ (resp. $\mathcal{M}_+(K)$). We denote by $\mathbbm{1}_S$ the indicator function of the set $S$.

\section{Main ingredients}\label{sec:POP}
\subsection{The classical Moment-SOS hierarchy}

Let $f\in\RR[\x]$ and $m\in \NN^*$ such that $g_1,\dotsc,g_m \in \RR[\x]$. Let $K$ be a basic closed semialgebraic set defined by
\begin{equation}
\label{eq:feasible set}
    K:=\set{\x\in \RR^n \mid\ g_1(\x)\geq 0,\dotsc,g_m(\x)\geq 0}.
\end{equation}
Let us consider the general Polynomial Optimization Problem (POP) of the form:
\begin{equation}
\label{eq:pop}
    {\bf P}: \quad f_{\min}=\inf_{\x\in K} f(\x).
\end{equation}
 Optimization problems of this type are generally nonlinear and nonconvex, which makes them very difficult to solve. However, one effective tool for tackling these problems is the Lasserre's hierarchy \cite{lasserre2001global}.
 The essence of Lasserre's hierarchy lies in transforming the initial problem into a hierarchy of finite-dimensional primal-dual semidefinite programming (SDP) problems.
 The primal problem revolves around characterizing moment sequences of Borel measures, forming a \textit{moment problem}. On the other hand, the dual problem involves describing positive polynomials that admit weighted sum of squares (SOS) decompositions, where the weights are the polynomials $g_j$ involved in the description of $K$. \\
 We recall that a polynomial $f\in\RR_{2d}[\x]$ is said to be SOS if there exist $k\in\NN^{*}$ and $p_1,\dots,p_k\in\RR_d[\x]$ such that $f=p_1^2+\cdots+p_k^2$. \par\smallskip

\noindent
Let $\y=(y_{\balpha})_{\balpha\in\NN^n}$ be a real-valued sequence. We can always associate to $\y$ the linear \emph{Riesz functional} $L_{\y}$ defined as follows:
\begin{align}
  L_{\y}: \RR[\x]\ni f\ & \mapsto L_{\y}(f):=\sum_{\balpha \in \supp(f)} f_{\balpha} y_{\balpha}\in  \RR \notag.
\end{align}
For $d\in \NN$, the symmetric matrix ${\bf M}_d(\y)$ of size $s(n,d)$ is called the \emph{pseudo-moment matrix} of order $d$ associated with $\y$, and its entry $(\balpha,\bbeta)$ is given by:
\begin{equation}
    {\bf M}_d (\y){(\balpha,\bbeta)}:=L_{\y}( \x^{\balpha}\x^{\bbeta}) = y_{\balpha+\bbeta} , \quad  \balpha, \bbeta \in \NN^n_d.
\end{equation}
Similarly, for $g\in \RR[\x]$, the symmetric matrix ${\bf M}_d(g\y)$ represents the \emph{localizing matrix} of order $d$ associated with $g$ and $\y$, defined as follows:
\begin{equation}
    {\bf M}_d (g\y){(\balpha,\bbeta)}:=L_{\y}(g \x^{\balpha}\x^{\bbeta}) = \sum_{\bgamma\in\supp(g)} g_{\bgamma}y_{\balpha+\bbeta+\bgamma} , \quad  \balpha, \bbeta \in \NN^n_{d-\lceil\operatorname{deg}(g)/2\rceil}.
\end{equation}
In particular, the localizing matrix of order $d$ associated with the constant polynomial $g_0=1$ is the moment matrix of order $d$. \\
Casting the POP as a moment problem follows from the following observation:
\begin{align}\label{eq: popasmom}
\begin{split}
        f_{\min} &= \inf_{\mu \in \mathcal{M}_+(K)} \set{\int_{\mathbb{R}^n} f(\x) \, d\mu(\x) {\: \big| \: \mu(K)=1}} \\
    &= \inf_{\y \in \RR^{\NN^n}} \left\{
    L_{\y}(f) \: \big| \: \exists\mu\in\mathcal{M}_+(K),\: {y_{\bf 0} = 1, \: y_{\balpha}=\int_{\RR^n}\x^{\balpha} \,d\mu(\x), \balpha\in\NN^n\setminus\set{\bf 0}}\right\}.
    \end{split}
\end{align}
Indeed, suppose that $\x^{*}\in\RR^n$ is a global minimizer of $f$ over $K$, and consider the Dirac measure $\mu^{*}=\delta_{\x^{*}}$. Then, $\mu^{*}$ is feasible for the optimization problem over measures in \eqref{eq: popasmom}, and $\displaystyle \int_{\mathbb{R}^n} f(\x) \, d\mu^{*}(\x)=f(\x^{*})=f_{\min}$, implying that $\displaystyle\inf_{\mu \in \mathcal{M}_+(K)} \set{\int_{\mathbb{R}^n} f(\x) \, d\mu(\x)\: \big| \: \mu(K)=1}\leq f_{\min}$. Conversely, let $\mu$ be feasible in \eqref{eq: popasmom}. Since for all $\x\in K,\: f(\x)\geq f_{\min}$, we obtain $\displaystyle\int_{\mathbb{R}^n} f(\x) \, d\mu(\x)\geq\int_{\mathbb{R}^n} f_{\min} \, d\mu(\x)=f_{\min}$, implying that $\displaystyle\inf_{\mu \in \mathcal{M}_+(K)} \set{\int_{\mathbb{R}^n} f(\x) \, d\mu(\x)\: \big| \: \mu(K)=1}\geq f_{\min}$. Equality between $f_{\min}$ and the sequence-based reformulation in \eqref{eq: popasmom} can be established in a similar way.
\par
The measure $\mu$ in the last line of expression \eqref{eq: popasmom} is referred to as the \textit{representing measure} of $\y$. This  reformulated optimization problem, while linear in $\y$, remains infinite-dimensional because it requires all possible moments of the measure $\mu$ to match the elements of the sequence $\y$.
However, the next result provides a way to relax this problem into a finite-dimensional one.
For each $j \in  \{1,\dots, m\}$, let us define $d_j:=\lceil\deg(g_j)/2\rceil$ and $d_{\min} :=\max\set{\lceil \deg(f)/2\rceil,d_1,\dotsc,d_m}$.
\begin{prop}[see \cite{Lasserre2009Book}]\label{prop: reprmeas}
    If $ \y \in \RR^{\NN^n_{2d}}$ is represented by a measure supported on $K$, then
\begin{enumerate}
    \item $ \MM_d(\y) \in \mathcal{S}_+^{s(n,d)}$,
    \item $ \MM_{d-d_j}(g_j \y) \in \mathcal{S}_+^{s(n, d-d_j)},\:  j \in  \{1,\dots, m\}  $.
\end{enumerate}
\end{prop}

\noindent
Proposition \ref{prop: reprmeas} facilitates constructing a hierarchy of moment relaxations for problem {\bf P}~\cite{lasserre2001global}.
Namely, the \emph{moment relaxation} for the POP \eqref{eq:pop} indexed by the \emph{relaxation order} $d\geq d_{\min}$ is defined as the following SDP:
\begin{align}
    \label{eq:primal}
    {\bf P}_d^{\operatorname{mom}}: \quad
    \begin{cases}
          \displaystyle f_d^{\operatorname{mom}}:=\inf_{\y \in \RR^{\NN^n_{2d}}}\quad & L_{\y}(f)\\
        \text{s.t.} \quad &{\bf M}_d(\y) \in \mathcal{S}_+^{s(n,d)},\\
        &  {\bf M}_{d-d_j}(g_j\y) \in \mathcal{S}_+^{s(n, d-d_j)}, \quad j\in \{1,\dotsc,m\},\\
        & y_{\mathbf 0}=1.
    \end{cases}
\end{align}

\noindent

\noindent
The sequence of SDP problems $(\mathbf{P}_d^{\operatorname{mom}})_{d \geq d_{\min}}$ is known as the \emph{moment hierarchy}. From now on, we assume that $K$ is compact and, without loss of generality, that its description includes a
redundant ball constraint  $R^2-\|\x\|_2^2\geq 0$ for some $R>0$. Then, the following classical convergence
result holds:

\begin{thm}[see \cite{lasserre2001global}]
The sequence of lower bounds $(f_d^{\operatorname{mom}})_{d\geq d_{\min}}$ is monotone non-decreasing and $\lim_{d\to+\infty} f_d^{\operatorname{mom}}=f_{\operatorname{min}}$.
\end{thm}

\noindent
Ball constraints also guarantee zero duality gap between \eqref{eq:primal} and its SOS-based SDP dual \cite{josz2016strong}, allowing us to denote the optimal value of the Moment-SOS hierarchy at order $d$ simply by $f_d$. In practice, we are often interested in obtaining finite convergence, i.e., $f_d = f_{\min}$ for some finite, preferably small, $d$. The following theorem provides sufficient conditions for this to occur:

\begin{thm}[Theorem~1.6, \cite{curto2000truncated}]\label{thm: flatness}
Consider the hierarchy of moment relaxations defined in \eqref{eq:primal}. If for some $d \geq d_{\min}$ there exists $d' \in \{d_{\min},\dots,d\}$  such that the SDP problem  ${\bf P}_d^{\operatorname{mom}}$ has an optimal solution $\y^{*}\in\RR^{\NN^n_{2d}}$ which satisfies
\begin{align}\label{eq: flatness}
    \mathrm{rank} \, {\bf M}_{d'}(\y^{*}) = \mathrm{rank} \, {\bf M}_{d' - d_{\min}}(\y^{*}),
\end{align}
then $f_d = f_{\min}$.  Moreover, the infinite-dimensional linear program from \eqref{prop: reprmeas} has an optimal solution $\mu^{*} \in \mathcal{M}_+(K)$, finitely supported on $r := \mathrm{rank} \, {\bf M}_d(\y^{*})$ global minimizers of $f$ on $K$.
\end{thm}
\noindent
If the rank condition~\eqref{eq: flatness} from Theorem~\ref{thm: flatness} is satisfied, then one can extract (some) global minimizers of the initial POP thanks to a numerical linear algebra procedure \cite{henrion2005detecting}.
\noindent
\begin{exmp}\label{ex: Leitmotif} Let $n=2$ and consider the following quadratic and nonconvex POP:
\begin{align}
        {\bf P}:\quad  \begin{cases} \displaystyle \min_{\x \in K} &f(\x):= -(x_1 - 1)^2 - (x_1 - x_2)^2 - (x_2 - 3)^2,\\
 \text{s.t.} & K:=\set{\x\in\RR^2 \: \mid \: g_j(\x)\geq 0,\: j\in\set{1,\dots,4}}, \text{where}\\
 & g_1(\x) := 1 - (x_1 - 1)^2,\:  g_2(\x) := 1- (x_1 - x_2)^2,\\
 & g_3(\x) :=1 - (x_2 - 3)^2,\:  g_4(\x) :=x_1- 0.3 x_2^2.
\end{cases}
    \end{align}
The first-order moment relaxation of this POP involves the pseudo-moment matrix of size $3$ and four linear localizing constraints.
Then, $\y^{*}_1=(1.0,1.6562,2.0833,3.3124,3.4061,4.4997)\in\RR^{\NN^2_2}$ is the optimal pseudo-moment sequence satisfying $L_{\y^{*}_1}(f) = f_1 \simeq -3$, and the corresponding first-order pseudo-moment matrix, whose numerical rank is 3, is given by:
\begin{align}\label{eq: inexactTRIANGLE}
    {\bf M}_1(\y^*_1)=\begin{bmatrix}
1.0000 & 1.6562 & 2.0833 \\
1.6562 &  3.3124 & 3.4061 \\
2.0833 & 3.4061 & 4.4997
\end{bmatrix}
\end{align}
Furthermore, after solving the second-order moment relaxation, which involves a pseudo-moment matrix of size $6$ and four localizing matrices of size $3$, we eventually obtain the optimal pseudo-moment sequence $\y^{*}_2=(1,2,2,4,4,4,8,8,8,8,16,16,16,16,16)\in\RR^{\NN^2_4}$ such that $\mathrm{rank} \,{\bf M}_{2}(\y^{*}_2) = \mathrm{rank} \, {\bf M}_{1}(\y^{*}_2)=1$, implying that $\y^{*}_2$ is actually the moment sequence of the Dirac measure concentrated at a global minimizer $\x^*=(2,2)$, and $L_{\y^{*}_2}(f)=f_2=-2=f_{\min}$.
\end{exmp}
\noindent
In \cite{FiniteCO-Nie2014}, it has been proven that the condition from \eqref{eq: flatness}, called \textit{flatness condition}, holds generically.
However, for many high-dimensional problems, solving higher order relaxations is either too costly or even impossible, since the size of the involved SDP matrices increases rapidly.
Indeed at a fixed order $d$, the corresponding moment relaxation involves a matrix of size $s(n,d)=\binom{n+d}{n} = \mathcal{O}(n^{d})$ when $n \to \infty$.
That is why we propose to exploit Christoffel-Darboux kernels to strengthen the relaxation bounds at order $d$.

\subsection{Christoffel-Darboux kernels}
\label{sec:cdk}
The Christoffel-Darboux kernel (CDK) is a well-known and valuable tool in the domain of approximation theory and orthogonal polynomials. Recently, empirical versions of this kernel have emerged, making  it quite useful for solving various tasks in the domain of data analysis. Namely, Christoffel-Darboux kernel provides a numerical framework for estimating the support of an unknown measure given only a finite number of its moments. For an in-depth discussion of the Christoffel-Darboux kernels, their properties, and historical developments, see \cite{nevai1986geza,simon2008christoffel,kroo2013christoffel,lasserre2022christoffel}.
\par
Let $\Omega \subset \RR^n$ be a compact set with non-empty interior, and let $\mu \in \mathcal{M}_+(\Omega)$.
{We highlight that $\Omega$ does not necessarily have to represent a feasible set of some POP.} The moment matrix of order $d\in\NN^{*}$ associated to $\mu$ is given by
\begin{align}
    {\bf M}_d(\mu):=\int_{\RR^n}\vb_d(\x)\vb_d(\x)^\top \, d\mu(\x) \in \mathcal{S}^{s(n,d)}_+,
\end{align}
where the integral should be understood component-wise.
Such measure $\mu$  induces a bilinear form on $\RR_d[\x]$, namely,  $\langle p,q\rangle_\mu=\int_{\Omega}p(\x)q(\x) \, d\mu(\x)$, for any $p,q \in \RR_d[\x]$.  If, in addition, $ \langle ,\rangle_\mu$ defines a scalar product then the finite-dimensional space $(\RR_d[\x],\langle ,\rangle_\mu) $ is as a Reproducing Kernel Hilbert space \cite{dunkl2014orthogonal}. Its associated kernel is called the \emph{Christoffel-Darboux kernel} (CDK), and it is defined for every $d\in\NN^{*}$ as
\begin{equation}\label{def:CDK}
    K_d^{\mu}:\ \RR^n\times \RR^n \ni (\x,\y) \mapsto K_d^{\mu}(\x,\y):=\displaystyle \sum_{\balpha\in \NN^n_d}p_{\balpha}(\x)p_{\balpha}(\y) \in\RR,
\end{equation}
where $(p_{\balpha})_{\balpha \in \NN^n_d}$ is a family of polynomials that are orthonormal with respect to $\langle,\rangle_\mu$.
Moreover, whenever $ \langle ,\rangle_\mu$ defines a valid scalar product, or stated equivalently, whenever the moment matrix ${\bf M}_d(\mu)$ is invertible,  it is possible to show \cite{simon2008christoffel} that
\begin{equation}
    K_d^{\mu}(\x,\y)=\vv_d(\x)^\top{\bf M}_d(\mu)^{-1}\vv_d(\y),\quad  \x,\y\in \RR^n.
\end{equation}
In this case, we can also define the \emph{Christoffel polynomial}  $\Lambda_d^{\mu}$ to be the diagonal of the CDK:
\begin{align}
\label{def:CF}
    \Lambda_d^{\mu}:\ \RR^d \ni \x \mapsto \Lambda_d^{\mu}(\x):=K_d^{\mu}(\x,\x)\in\RR_{+}.
\end{align}
By construction, $\Lambda_d^{\mu}=\displaystyle\sum_{\balpha\in \NN^n_d} p_{\balpha}^2$ is an SOS polynomial of degree $2d$.
\begin{rem}
    Notice that the Christoffel-Darboux kernel and, therefore, the Christoffel polynomial itself, only depend on the finite sequence $\y\in\RR^{\NN^n_{2d}}$ of moments of $\mu$. Thus, we can interchangeably write $\Lambda_d^{\mu}$ or $\Lambda_d^{\y}$,  $K_d^{\mu}$ or $K_d^{\y}$, and ${\bf M}_d(\mu)$ or ${\bf M}_d(\y)$.
\end{rem}
\noindent
One of the most remarkable features of the Christoffel polynomials $(\Lambda_d^\mu)_{d \in \mathbb{N}^{*}}$ is their ability to identify the support of the underlying measure $\mu$. Indeed, as explained in \cite[Section 4.3, p.
50–51]{ChristoffelBook-Lasserre2022}, when increasing the degree $d$, one may observe the following important dichotomy:
\begin{itemize}
    \item $\forall \x \in \text{supp}(\mu), \ \Lambda_d^\mu(\x)$ grows at most polynomially when we let $d \to +\infty$;
      \item $\forall \x \notin \text{supp}(\mu), \ \Lambda_d^\mu(\x)$ grows at least exponentially  when we let $d \to +\infty$.
\end{itemize}
\noindent
Then, $\text{supp}(\mu)
$ can be
approximated via appropriately chosen sublevel sets of the form:
\begin{align}
    S_d(\mu,\gamma):=\set{\x\in\RR^n \mid \Lambda^{\mu}_d(\x)\leq \gamma}.
\end{align}
To illustrate this phenomenon, let $\mu$ be the uniform measure over the unit square $[0, 1]^2$.
\begin{figure}[!htbp]
    \centering
    \begin{subfigure}[b]{0.495\textwidth}
        \centering
\includegraphics[width=\textwidth,height=0.75\textwidth]{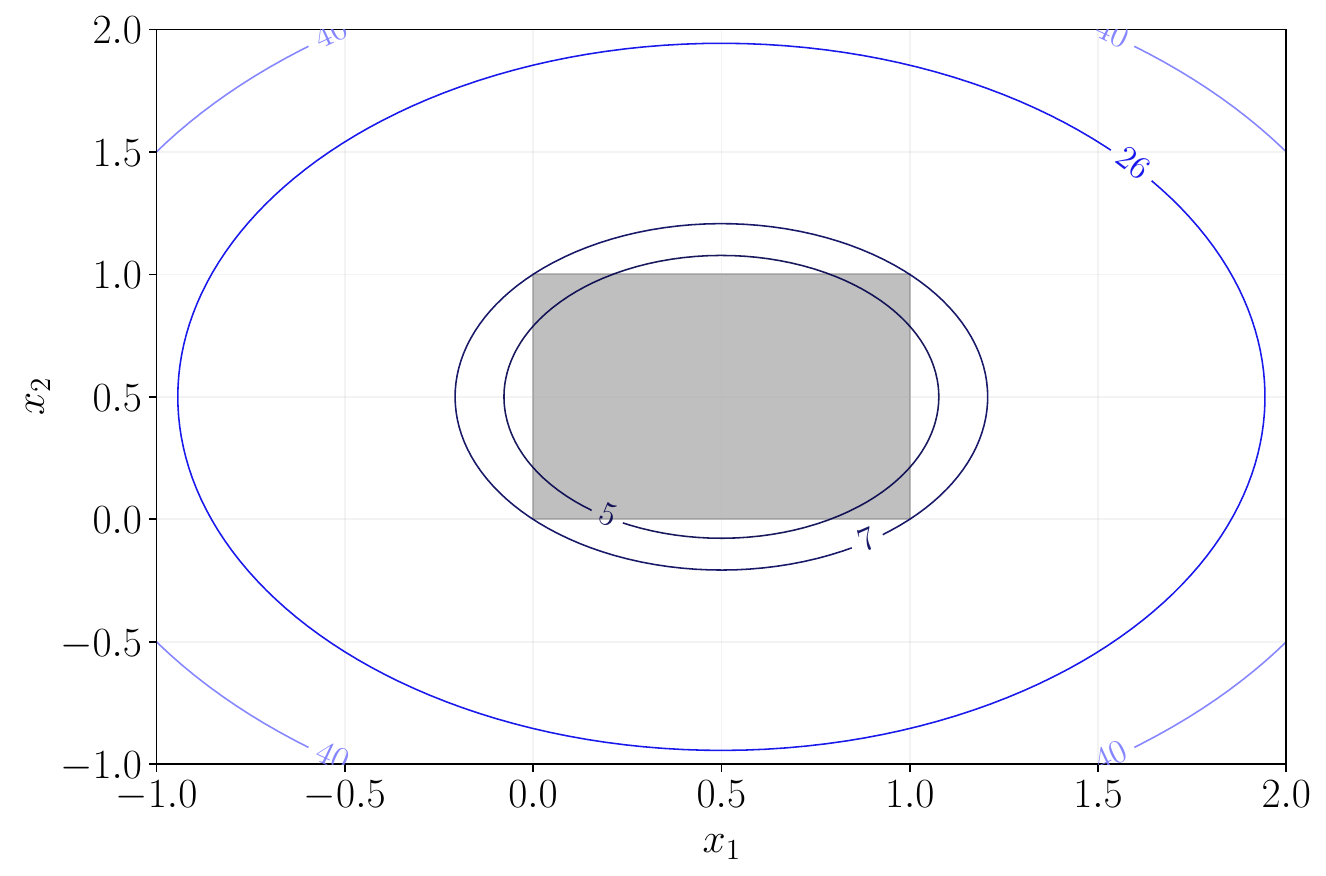}
        \caption{$d=1$}
    \end{subfigure}
    \hspace{-0.15cm}
    \begin{subfigure}[b]{0.495\textwidth}
        \centering
        \includegraphics[width=\textwidth,height=0.75\textwidth]{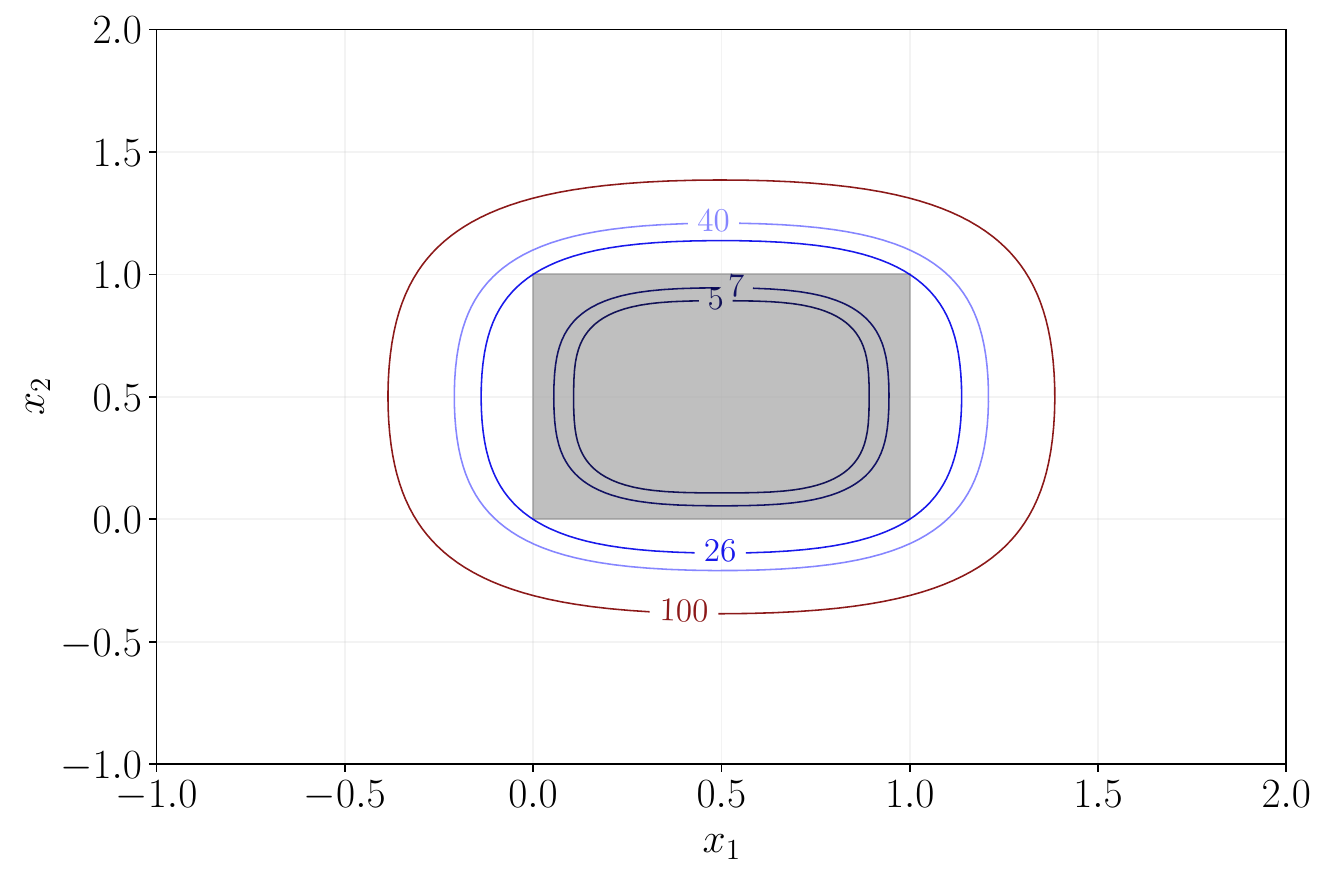}
        \caption{$d=2$}
    \end{subfigure}
    \caption{Sublevel sets of the Christoffel polynomials associated with the uniform measure over the unit square, i.e., $d\mu(\x)=\mathbbm{1}_{[0,1]^2}(\x)d\x$. The figure depicts $S_1(\mu,\gamma)$ (left) and $S_2(\mu,\gamma)$ (right) for $\gamma \in \{5, 7, 26, 40, 100\}$.}
    \label{fig: supportidentification}
\end{figure}
\noindent
As displayed in Figure \ref{fig: supportidentification}, the sublevel sets $S_2(\mu,\gamma)$ constructed using quartic Christoffel polynomials, provide better approximation of the unit square. Furthermore, the exponential growth of $\Lambda^\mu_d(\x)$ for points outside of the unit square can be illustrated by, for example, considering the point $\x=(0,1.25)$, which belongs to $S_1(\mu,26)$ but does not even belong to $S_2(\mu,100)$.

\section{Strengthening relaxation bounds}
\label{sec:Heuristics}
\noindent
When solving any POP instance, quite often, the main goal is to recover the optimal measure $\mu^{*}$, which is finitely supported on the minimizers of $f$, as described in Theorem \ref{thm: flatness}. However, for difficult POPs, the relaxation bound obtained at the step $d$ of the hierarchy is usually not exact, i.e., $f_d<f_{\min}$. In addition, elevated computational costs can prevent one from achieving convergence by computing higher-order bounds such as $f_{d+1}$ or $f_{d+2}$, for example.

\par
In this section, we present two heuristic methods that leverage the information available at order $d$ of the hierarchy. More precisely, these methods aim to construct a modified POP and \textit{solve} its relaxation at the same order $d$.
Since this modified POP restricts the feasible set, this approach yields tighter bounds. Namely, we will set the new feasible set to be
\begin{align}
    \widetilde{K}:=K\cap S_d(\y^{*},\gamma)=K\cap \set{\x\in\RR^n \, \mid \, \gamma - \Lambda_d^{\y^*}(\x)\geq 0 },
\end{align}
where $\y^*$ is the optimal solution of the initial relaxation at order $d$, and $\gamma>0$ is the parameter to be tuned. In what follows, we argue that this approach can \textit{often} help in obtaining a tighter lower bound $\widetilde{f}_d$, such that $f_d \leq \widetilde{f}_d \leq f_{\min}$, and may even enable the extraction of minimizers, \textit{already at step $d$ of the hierarchy}. Before delving into a detailed description of our method, we first outline some of the general challenges associated with it.\\
\textit{Firstly}, if the measure~$\mu$ under consideration is atomic, which is typically the case for POP, the associated moment matrix~${\bf M}_d(\mu)$ is inevitably rank deficient, since the support of~$\mu$ lies on an algebraic variety. Consequently, computing~$\Lambda_d^{\mu}$ as described in~\eqref{def:CF} may introduce significant numerical inaccuracies. To address these issues, we compute the Christoffel polynomials in a slightly modified manner. Observe that we can always write
\begin{align}
    \label{eq: spectral}
    {\bf M}_d(\mu)={\bf P}{\bf E}{\bf P}^\top,
\end{align}
where ${\bf P}$ is an orthonormal matrix of size $s(n,d)$ whose columns are denoted by $\p_i, \:i\in\set{1,\dots,s(n,d)}$, and ${\bf E}\in\mathcal{S}_+^{s(n,d)}$ is the diagonal matrix containing the eigenvalues  $e_{1}\geq\dots \geq e_i \geq \dots \geq e_{s(n,d)}\geq 0$ of
${\bf M}_d(\mu)$. Interpreting the eigenvectors $\p_i$ as coefficients of polynomials $p_i \in \RR_d[\x]$ yields the \textit{Tikhonov regularization} of the Christoffel polynomial of order $d$:
\begin{align}\label{eq: regularization}
\widetilde{\Lambda}_d^{\mu}:=\sum_{i=1}^{s(n,d)}\frac{p_i^2}{e_i+\beta}\approx \vb_d^\top\left({\bf M}_d(\mu)+\beta {\bf I}_{s(n,d)} \right)^{-1}\vb_d,
\end{align}
    where $\beta>0$ is a small regularization parameter, and ${\bf I}_n$ denotes the identity matrix of size $n$. A measure-theoretic interpretation of this regularization can be found in \cite{CDK-Swan2021}.
\noindent
For instance, if we express ${\bf M}_d(\mu)$  in the basis of polynomials $(q_{\balpha})_{\balpha}$ that are orthonormal w.r.t. the Lebesgue measure $\lambda$ on $K$, and if $\y_\lambda$ is the moment sequence of $\lambda$, then $\widetilde{\Lambda}^{\mu}_d=\Lambda^{\mu+\beta\lambda}_d$, and we have
\begin{align}
        L_{\y+\beta \y_\lambda}(p)&=L_{\y}(p)+\beta\,\int_K p(\x)\,d\lambda(\x)\, ,\quad\forall p\in\mathbb{R}_{2d}[\x],\\
{\bf M}_d(\y+\beta\,\y_\lambda)(\balpha,\bgamma)&= L_{\y}(q_{\balpha} q_{\gamma})\,+\beta\,\mathbbm{1}_{\balpha=\bgamma}\,,\quad \forall\balpha,\bgamma\in\mathbb{N}^n_d.
\end{align}
\noindent

By studying the limit as the regularization parameter tends to zero (see Lemma 4 in \cite{Pauwels2021}), one can justify that the expression defined in \eqref{eq: regularization} retains its support dichotomy properties. In practice, $\widetilde{\Lambda}_d^{\mu}$ must be approximated by $\widetilde{\Lambda}_d^{\y^{*}}$, since we only have access to an optimal pseudo-moment vector $\y^{*}\in\NN^n_{2d}$ associated with the measure under consideration. Fortunately, as stated in Theorem \ref{thm: pseudoapprox}, whenever $\y^{*}$ is sufficiently close to the true moment sequence, then the error introduced by this approximation is small.
\begin{thm}[Lemma 1, \cite{CDK-Swan2021}]\label{thm: pseudoapprox}
Let $ \displaystyle \alpha := \left\| \mathbf{I}_{s(n,d)} - \mathbf{M}_d(\mu + \beta \lambda)^{\frac{1}{2}} \, \mathbf{M}_d(\y^{*} + \beta \y_\lambda)^{-1} \, \mathbf{M}_d(\mu + \beta \lambda)^{\frac{1}{2}} \right\|$
where $\|\cdot\|$ denotes the matrix operator norm. Then,
$ \displaystyle
\sup_{\x \in \mathbb{R}^n} \left| 1 - \frac{\widetilde{\Lambda}_d^{\y^{*}}(\x)}{\widetilde{\Lambda}_d^{\mu}(\x)} \right| \leq \alpha.$
\end{thm}

\noindent With this in mind, let $r\in\NN$ the number of zero eigenvalues, or, stated differently, the dimension of the kernel of the pseudo-moment matrix ${\bf M}_d(\y^{*})$. Then, the sublevel sets\footnote{The set in \eqref{eq: levelsetsREG} is a proper subset of the sublevel set $\set{\x \in \RR^n \mid \widetilde{\Lambda}_d^{\y^*}(\x) \leq \gamma + r}$. Although this constitutes a slight abuse of terminology, we refer to it as a sublevel set for the sake of intuition.
} associated to $\widetilde{\Lambda}_d^{\y^*}$ that we will use for strengthening the relaxations can be written in the following way:
\begin{align}\label{eq: levelsetsREG} \widetilde{S}_d(\y^{*},\gamma):=\set{\x\in\RR^n \: \Bigg| \: \sum_{i=1}^{s(n,d)-r}\frac{p_i^2(\x)}{e_i+\beta}\leq \gamma, \: p_j^2(\x)\leq \beta,\: s(n,d)-r+1\leq j \leq s(n,d)}.
\end{align}
We give the rationale behind the expression in \eqref{eq: levelsetsREG}: given $\x \in \mathbb{R}^n$ such that $p_j^2(\x) > 0$ for some $j \in \set{ s(n,d) - r + 1, \dots, s(n,d) }$, division by the corresponding eigenvalue $e_j + \beta$ yields a large positive contribution. Consequently, such an $\x$ would not belong to the $\gamma$-sublevel set for any reasonably large value of $\gamma$, making it unlikely to lie in the support of the optimal measure we aim to recover. Thus, in light of Proposition 1 from \cite{Pauwels2021}, it is reasonable to expect that the minimizers of $f$, or, in view of Theorem \ref{thm: flatness}, the atoms of the optimal measure, lie in a neighborhood of the algebraic variety $\set{\x \in \RR^n \mid p_j^2(\x) = 0,\: j \in \set{s(n,d) - r + 1, \dots, s(n,d)}}$.
Moreover, the sum involving polynomials outside the kernel of the pseudo-moment matrix identifies more precisely the regions within this variety where the minimizers may lie.\\
\textit{Secondly}, the challenge lies in appropriately selecting the threshold $\gamma>0$. If $\gamma$ is too small, the resulting feasible set may shrink excessively. On the other hand, if $\gamma$ is too large, the relaxation may remain too loose, failing to improve the bound in a meaningful way. To illustrate this more effectively, let us revisit Example \ref{ex: Leitmotif}.
  \begin{figure}[!htbp]
    \centering    \includegraphics[width=0.68\textwidth]{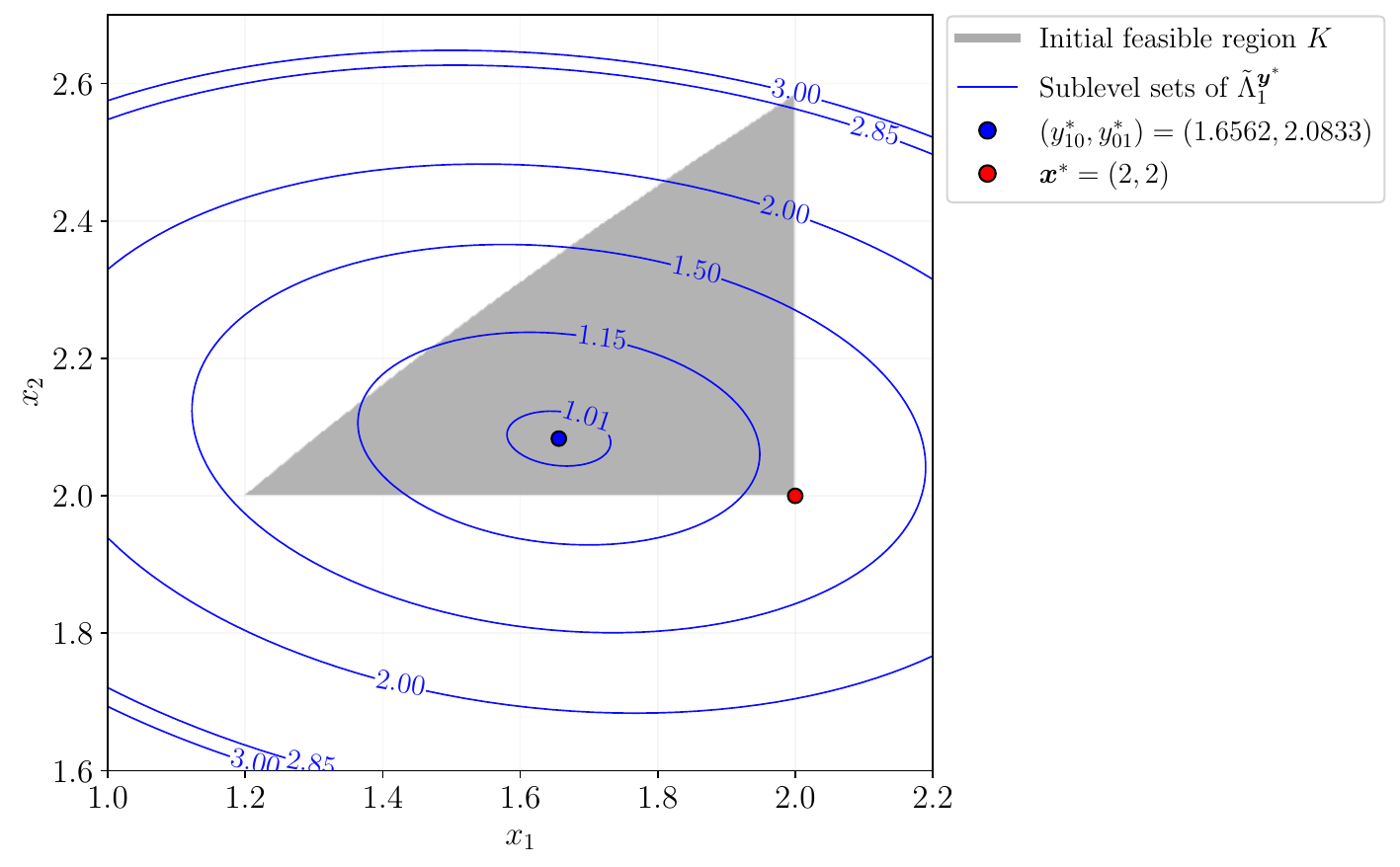}
    \caption{Depicting sublevel sets $\widetilde{S}_1(\y^{*},\gamma)$ associated to the POP from  Example \ref{ex: Leitmotif}, where $\gamma\in\set{1.01, 1.15, 1.50, 2.0, 2.85, 3.0}$. The red point is the true minimizer of $f$, and the blue point corresponds to the pseudo-moments of order one extracted from the optimal solution $\y^*$ of the first-order moment relaxation.}
    \label{fig:choosing_gamma}
\end{figure}
\noindent
As illustrated in Figure \ref{fig:choosing_gamma}, different sublevel sets impose different restrictions on the initial feasible set $K$. For instance, with $\gamma = 1.50$, the upper part of $K$ is removed, yielding a \textit{valid and improved} lower bound $\widetilde{f}_1$ such that $f_{\min} = -2 > \widetilde{f}_1 = -2.3131 > f_1 = -3$. In contrast, if $\gamma$ is too small, the feasible set may become overly restrictive and exclude the true minimizer of $f$, resulting in an \textit{invalid} bound. For example,  $\widetilde{f}_1 = -1.8577 > f_{\min}$ if $\gamma = 1.15$.
Here, \textit{invalid} means that $\widetilde{f}_1$ satisfies  $\widetilde{f}_1 > f_{\min}$, i.e., it becomes an upper bound, which we want to avoid.
\begin{rem}
A natural question is why one would prefer Christoffel polynomial constraints over a simple ball constraint centered at a known local solution $\overline{\x}\in K$. The reason is that using \eqref{eq: levelsetsREG} \textit{requires no external information} to infer the potential locations of minimizers. Stated differently, the method is entirely relaxation-dependent and exploits the fact that SDP relaxations typically provide good quality bounds, meaning that the available pseudo-moment matrix can be expected to be not too far from the optimal moment matrix. Moreover, the sublevel sets associated with Christoffel polynomials can take highly nontrivial shapes, offering significantly greater flexibility than standard ball constraints. As a result, feasible set refinement can be performed in a more targeted and nuanced manner.
\end{rem}
\noindent
To summarize, level sets of the Christoffel polynomial of order $d$ give us valuable information about the location of the true minimizers of $f$.

To leverage this information, efficient control of the parameter $\gamma$ is essential. The following subsections present two heuristic methods to balance lower bound tightness with the risk of obtaining an upper bound.
\subsection{Iterative approach - \texorpdfstring{$\HH$}{H1}}
The first heuristic method we propose, called $\HH$, is aimed at progressively reducing the feasible set until a more satisfactory lower bound is obtained. This approach relies on re-executing relaxations of order $d$, which is often more efficient than increasing the relaxation order in the hierarchy, particularly if the number of re-executions or iterations remains relatively small. At each iteration $k\geq 1$, the $\gamma_k$-sublevel set is selected in such a way that the previously optimal pseudo-moment sequence $\y^{*}_{(k-1)}\in\RR^{\NN^n_{2d}}$ becomes infeasible for the modified POP, forcing the lower bound to improve.
Choosing the threshold $\gamma_k$ which ensures this behavior is based on the following proposition:
\begin{prop}
Let $d \in \NN^{*}$ and $\y^*\in \RR^{\NN^n_{2d}}$ be an optimal solution of the moment relaxation of order $d$.
Then,
\begin{align}
L_{\y^{*}}\left(\widetilde{\Lambda}^{\y^*}_d\right)
\,=\,\sum_{i=1}^{s(n,d)-r}\frac{e_i}{e_i+\beta}\quad(\approx \mathrm{rank}\:{\bf M}_d(\y^{*})\mbox{ when $\beta$ is small)}.
\end{align}
\end{prop}
\begin{proof}

From \eqref{eq: spectral} and
\eqref{eq: regularization} one has
\begin{align}
    \begin{split}
        L_{\y^{*}}\left(\widetilde{\Lambda}^{\y^*}_d\right) &
                = L_{\y^{*}} \left( \sum_{i=1}^{s(n,d)}\frac{p_i^2}{e_i+\beta}\right)=\sum_{i=1}^{s(n,d)}\frac{1}{e_i+\beta} L_{\y^{*}} \left( p_i^2\right)\\
        &=\sum_{i=1}^{s(n,d)}\frac{1}{e_i+\beta} \p_i^\top {\bf M}_d(\y^{*})\p_i = \sum_{i=1}^{s(n,d)}\frac{1}{e_i+\beta}e_i\mathbbm{1}_{\p_i\notin\operatorname{ker}{\bf M}_d(\y^{*})}.
    \end{split}
\end{align}
As the dimension of the kernel of the pseudo-moment matrix ${\bf M}_d(\y^{*})$ is $r\in\NN^{*}$, one obtains
\begin{align}
    L_{\y^{*}}\left(\widetilde{\Lambda}^{\y^*}_d\right)
    \,=\,\sum_{i=1}^{s(n,d)-r}\frac{e_i}{e_i+\beta}\,,
\end{align}
and therefore $L_{\y^{*}}\left(\widetilde{\Lambda}^{\y^*}_d\right)\approx\mathrm{rank}({\bf M}(\y^*))$ when $\beta$ is small,
which concludes the proof.
\end{proof}
\noindent

Thus, by choosing $\gamma \geq L_{\y^{*}}(\widetilde{\Lambda}^{\y^*}_d)$, minimizing $f$ over $\widetilde{K} = K \cap \widetilde{S}_d(\y^*, \gamma)$ yields little improvement, as the existing optimal solution $\y^{*}$ remains feasible. Instead, using the penalized threshold $\gamma = (1-\varepsilon)L_{\y^{*}}(\widetilde{\Lambda}^{\y^*}_d)$ with $\varepsilon \in (0,1)$ excludes $\y^{*}$ from the feasible set, improving the lower bound. This process is repeated until a maximum number of iterations is reached or the bound improvement is deemed satisfactory. The true optimal value $f_{\min}$ being unknown, we assess the tightness of the new bound $\widetilde{f}_d$ by computing the \textit{(relative) optimality gap}:
\begin{align}
\Delta(f(\overline{\x}), \widetilde{f}_d) := \left(\frac{|f(\overline{\x}) - \widetilde{f}_d|}{|f(\overline{\x})|}\mathbbm{1}_{f(\overline{\x}) \neq 0} + |f(\overline{\x}) - \widetilde{f}_d|\mathbbm{1}_{f(\overline{\x}) = 0}\right) \times 100\%,
\end{align}
where $f(\overline{\x})$ is an upper bound on $f_{\min}$ obtained by evaluating $f$ at a given locally optimal solution $\overline{\x} \in \RR^n$. The steps of this iterative approach $\HH$ are outlined in Algorithm~\ref{alg:iteration-based-heuristic}.
\begin{algorithm}[!htbp]
\caption{Implementing $\HH$ }
\label{alg:iteration-based-heuristic}
\begin{itemize}[label={}]
\item {\bf Input:} Relaxation order $d$, gap tolerance $\delta > 0$, maximum number of iterations $N$, and a penalization factor $\varepsilon \in (0,1)$.
\item {\bf Initialize:} \(k=0\)
\begin{enumerate}
\item Initialize the feasible set $\widetilde{K}_k=K$.
\item Solve the moment relaxation of order $d$. Recover its optimal solution \( \y^{*}_{k}\), relaxation bound $\widetilde{f}_{d,k}=f_d$, and a local solution $\overline{\x}_{k}$ yielding an upper bound $\operatorname{ub}_k=f(\overline{\x}_{k})$.
\item Initialize the relative gap $\Delta_k=\Delta\left(f(\overline{\x}_k),\widetilde{f}_{d,k}\right)$.\\
\end{enumerate}
\item {\bf While} $k<N$
\begin{enumerate}[resume]
\item \textbf{If} $\widetilde{f}_{d,k} > \operatorname{ub}_k$  \textbf{or} $\Delta_{k} \leq \delta $
    \begin{itemize}
        \item \textbf{Break}.
    \end{itemize}
    \item Compute the Christoffel polynomial $\widetilde{\Lambda}^{\y^*_k}_d$ and set $\gamma_{k} = L_{\y^{*}_k}\left( \widetilde{\Lambda}_d^{\y^*_k}\right)$.
    \item \label{alg: setModification} Modify the feasible set via $\widetilde{K}_{k+1}=\widetilde{K}_k\cap\widetilde{S}_d(\y^{*}_k,(1-\varepsilon)\gamma_k)$, and solve the moment relaxation of order $d$. Recover the new relaxation bound $\widetilde{f}_{d,k+1}$, the new optimal solution $\y^{*}_{k+1}$, and the new locally optimal solution $\overline{\x}_{k+1}$.
    \item Set $\operatorname{ub}_{k+1}=\displaystyle\min_{i\in\set{0,k+1}}\set{f(\overline{\x}_{i})}$ and re-evaluate the gap $\Delta_{k+1}=\Delta\left(\operatorname{ub}_{k+1}, \widetilde{f}_{d,k+1} \right)$.
    \item Update the iteration count $k=k+1$.
\end{enumerate}
\item \textbf{Output}: Sequence of bounds $\left(\widetilde{f}_{d,k}\right)_{k\geq1}$ satisfying $\widetilde{f}_{d,k+1}\geq \widetilde{f}_{d,k}\geq f_d$.
\end{itemize}
\end{algorithm}
\FloatBarrier
\noindent
We emphasize that, in general, there are no guarantees that Algorithm \ref{alg:iteration-based-heuristic} will produce a sequence $(\widetilde{f}_{d,k})_{k \geq 1}$ such that its last element is a valid lower bound on $f_{\min}$. However, our experimental results presented in Section \ref{sec:num_exp} suggest that such  \textit{non-desirable} occurrences are relatively rare.
\begin{rem}\label{rem: retaining}
  Notice from step \eqref{alg: setModification} of Algorithm \ref{alg:iteration-based-heuristic} that, at each iteration, all previously added sublevel constraints are retained.
  This ensures that the feasible set of POP is progressively restricted, and, in turn, that no previously excluded optimal pseudo-moment sequences re-enter the feasible set of moment relaxations.
\end{rem}
\noindent

In view of Remark \ref{rem: retaining}, we are able to derive the following separation result:
\begin{thm}[Uniform separation of $\HH$ iterates]\label{thm:H1-separation}
Fix $\beta>0$ and $\varepsilon\in(0,1)$. Let $\y_k^*$ and $\y_j^*$, $j>k$, denote the optimal pseudo-moment sequences obtained at iterations $k$ and $j$ of Algorithm~\ref{alg:iteration-based-heuristic}. Then
\begin{align}\label{eq:H1-separation}
\left\|{\bf M}_d(\y_j^*)-{\bf M}_d(\y_k^*)\right\|_F
\ge
\frac{\varepsilon\beta}{(1+\beta)\sqrt{s(n,d)}}.
\end{align}
\end{thm}
\noindent
The proof is given in \hyperref[app:H1]{Appendix~A.1}. Since the standing ball constraint makes
the set of normalized feasible moment matrices bounded, the
Bolzano-Weierstrass theorem implies that any infinite sequence of feasible
iterates would admit a convergent subsequence. The uniform separation
in~\eqref{eq:H1-separation} rules this out. In other words, $\HH$ necessarily terminates after \textit{finitely many} iterations.
\par
Uniform separation describes the successive SDP solutions, but it does not ensure
that the added restrictions preserve a global minimizer.
\par\noindent
However, on a particular class of $0$--$1$ knapsack problems, we are able to provide a more explicit analysis. Namely, the successive $\HH$ restrictions preserve a
globally optimal binary point while reducing the pseudo-moment weight of
a fractional endpoint. This yields geometric convergence of the first-order
bounds; see \hyperref[app:H1-class]{Appendix~A.2}.

\par
\noindent
The behavior of Algorithm \ref{alg:iteration-based-heuristic} is illustrated on Example \ref{ex: Leitmotif}, with results shown in Figure \ref{fig: iterations}.

\begin{figure}[!htbp]
    \centering
    \begin{subfigure}[b]{0.40\textwidth}
        \centering
\includegraphics[width=\textwidth,height=\textwidth]{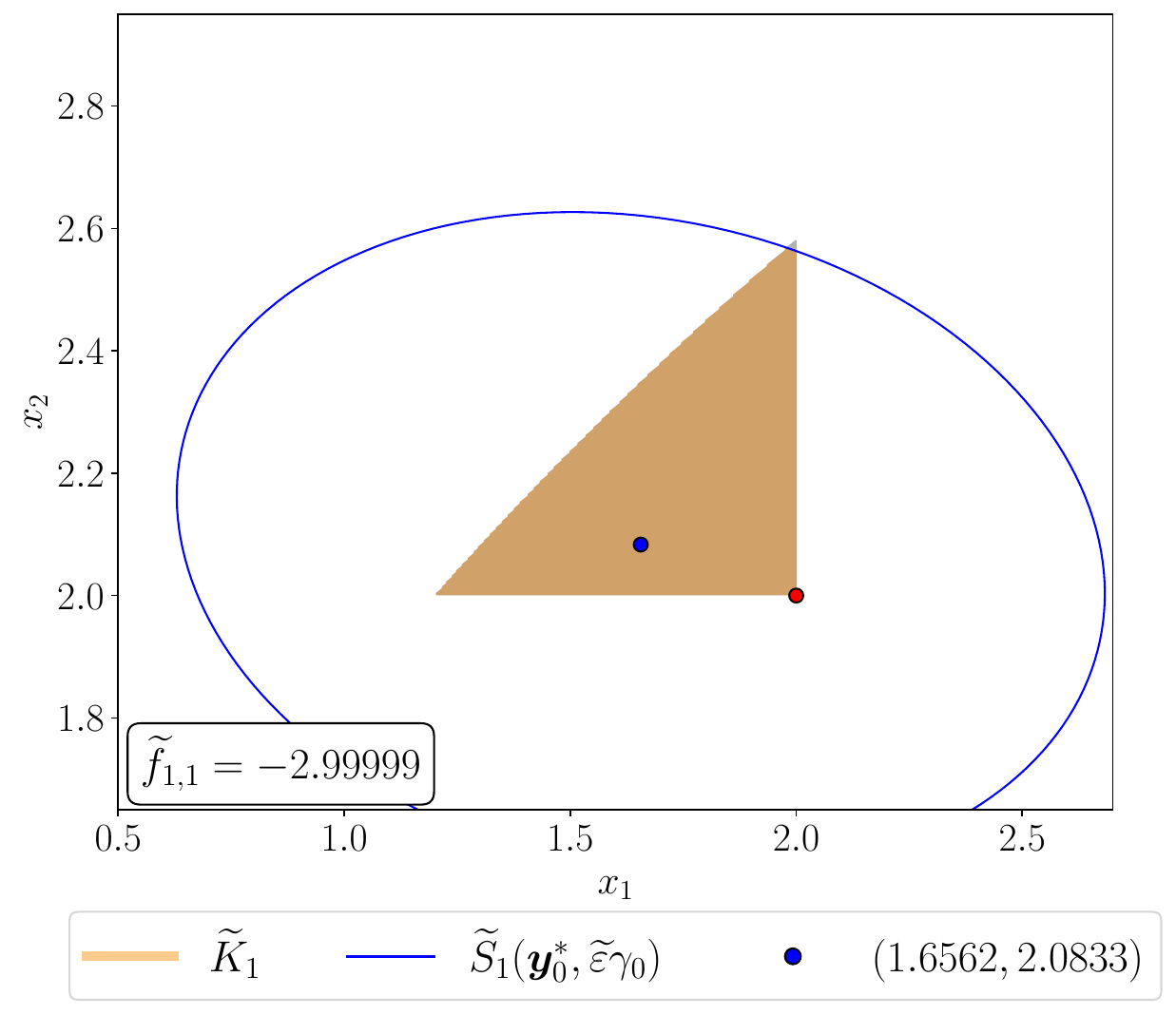}
    \end{subfigure}
    \hspace{-0.2cm}
    \begin{subfigure}[b]{0.40\textwidth}
        \centering
\includegraphics[width=\textwidth,height=\textwidth]{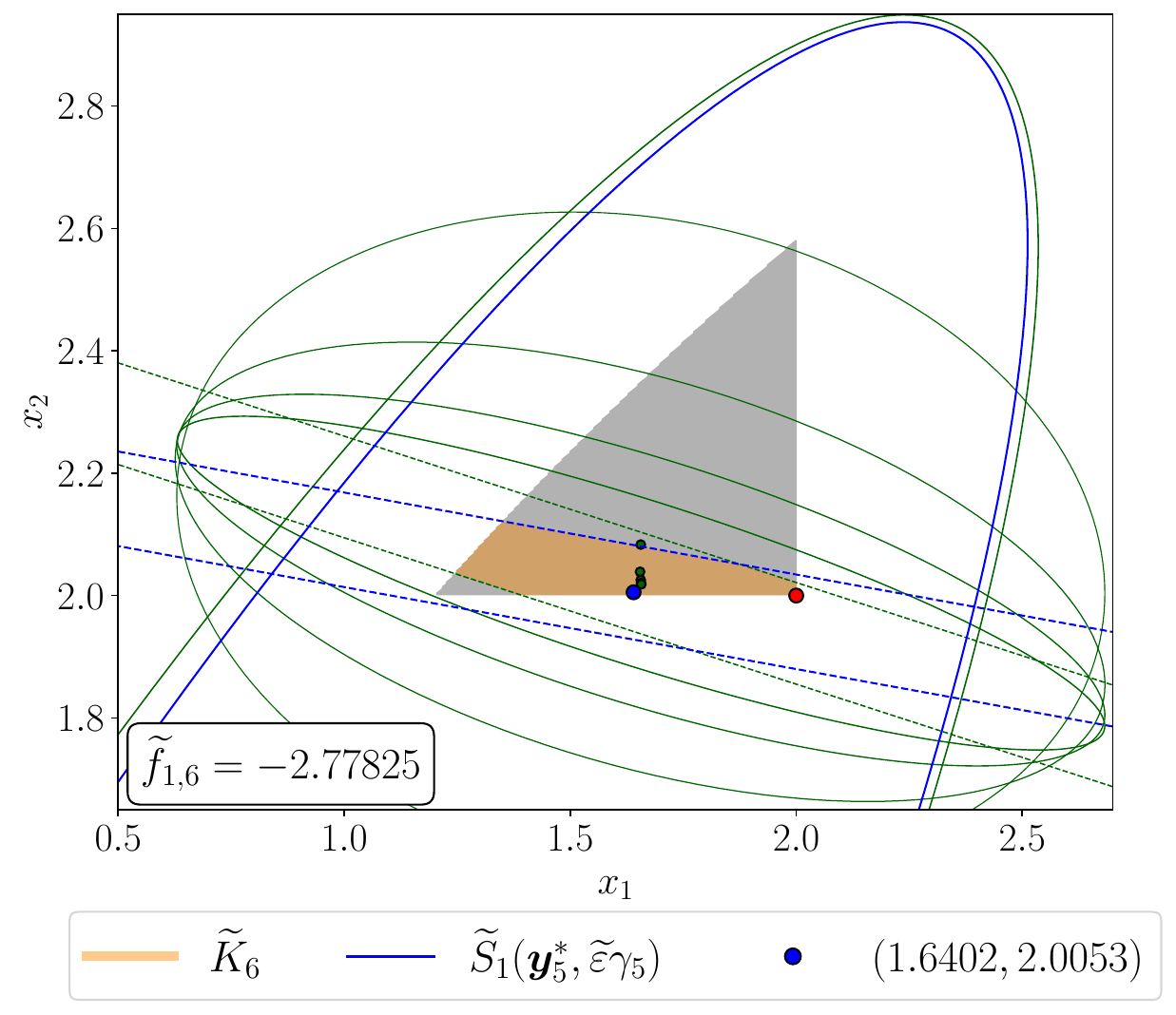}
    \end{subfigure}
    \\
    \begin{subfigure}[b]{0.40\textwidth}
        \centering
\includegraphics[width=\textwidth,height=\textwidth]{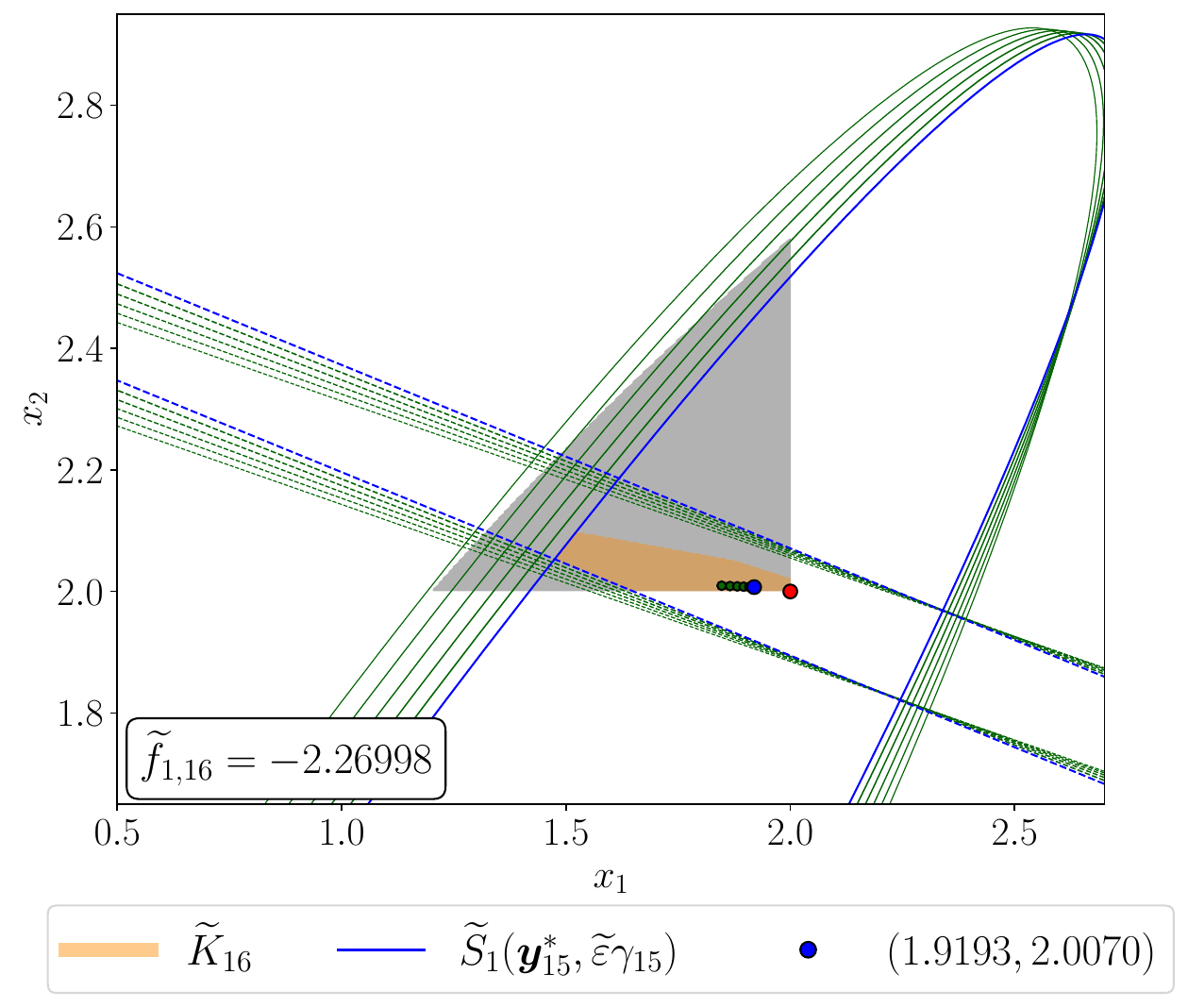}
    \end{subfigure}
    \hspace{-0.2cm}
    \begin{subfigure}[b]{0.40\textwidth}
        \centering
\includegraphics[width=\textwidth,height=\textwidth]{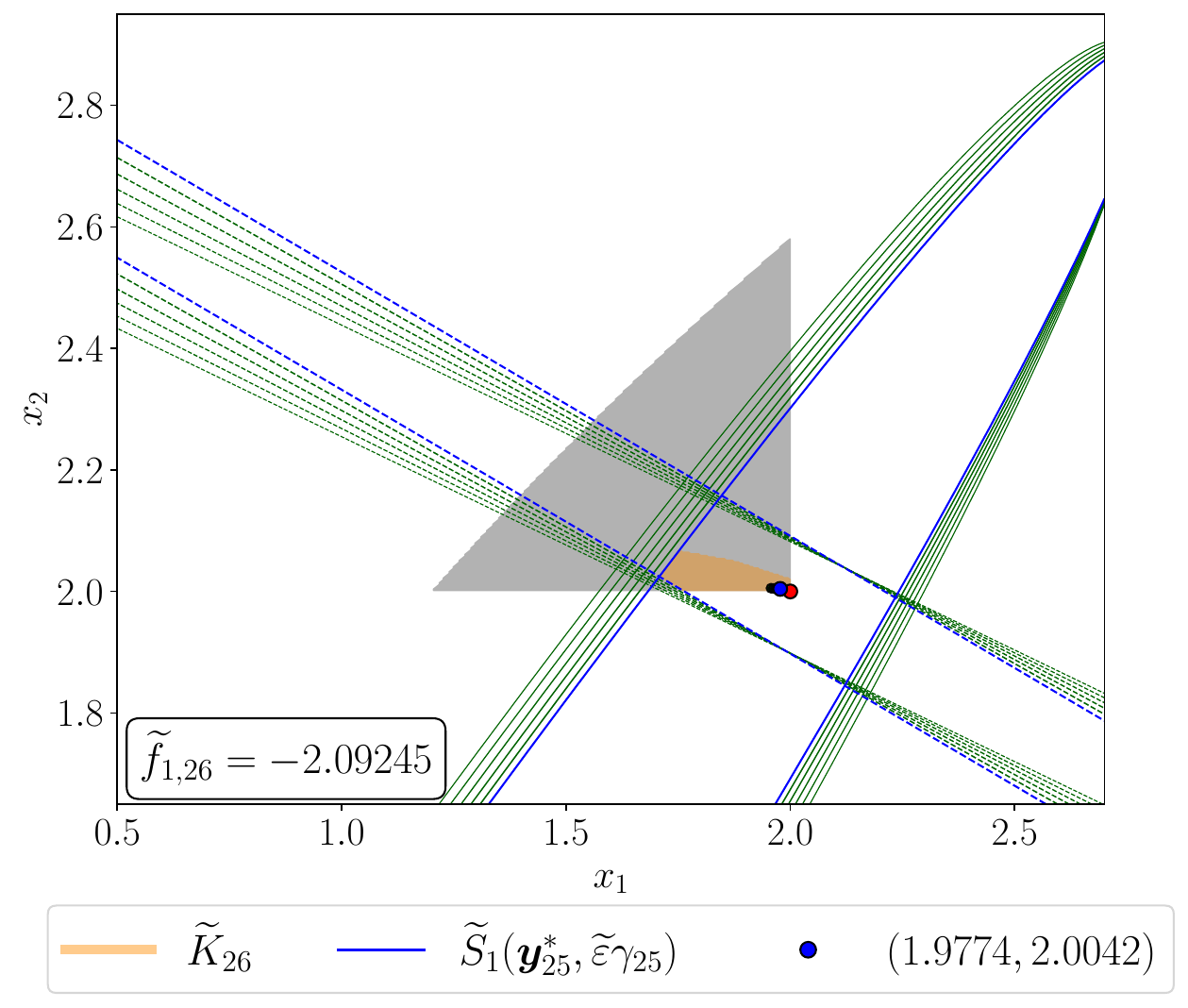}
    \end{subfigure}
    \caption{Iterations $k\in\set{0,5,15,25}$ of $\HH$, with $\varepsilon=0.05$ and $\widetilde{\varepsilon}=1-\varepsilon$.}
    \label{fig: iterations}
\end{figure}
\FloatBarrier
 \noindent
 Let us provide a detailed description of the Figure \ref{fig: iterations}. The red point $\x^{*} = (2,2)$ is the global minimizer. Green curves correspond to the level sets of Christoffel polynomials from the preceding 5 iterations,
 while blue lines indicate the level sets for the current iteration. If the $3 \times 3$ pseudo-moment matrix is of full rank, only solid lines are present. In contrast, dashed lines represent the contribution of polynomials that belong to the kernel of the corresponding pseudo-moment matrix. Green points represent the pseudo-moments of order one associated with the 5 preceding iterations, whereas the blue point depicts the pseudo-moments of order one for the current iteration. Notice that the Christoffel-based constraints incrementally improve the relaxation bound from $\widetilde{f}_{1,1} = -2.99999$ to $\widetilde{f}_{1,26} = -2.09245$ by \textit{informed} refinement of the feasible sets $\widetilde{K}_k$ (shown in orange). Moreover, the shrinking sublevel sets $\{\widetilde{S}_1(\y_k^{*}, 0.95\gamma_k)\}_{0 \leq k \leq 25}$ reflect a decrease in uncertainty about the minimizer's location. Therefore, we observe convergence of the first-order pseudo-moments $\left(y^{*}_{k,10}, y^{*}_{k,01}\right)_{0 \leq k \leq 25}$ toward those of the Dirac measure at $(2,2)$, which is consistent with Corollary 3.5. from \cite{Schweighofer2005pseudomom}, where the same result is achieved by increasing the relaxation order.

\subsection{Local-solution-based approach - \texorpdfstring{$\hh$}{H2}}

Unlike the previous heuristic, where the reduction of the feasible set was entirely guided by the information extracted from the relaxation itself, our second heuristic, named $\hh$, explicitly leverages some available solution $\overline{\x}\in\RR^n$ which is \textit{locally optimal}, to determine the shape of the new feasible set. This procedure involves constructing \textit{univariate} or \textit{marginal} Christoffel polynomials to assess the quality of the given local solution and to selectively add  constraints to the initial POP, thereby strengthening the lower bound of the moment relaxations.
\par
Let $i\in\set{1,\dots,n}$.  We define $\NN^{n{[i]}}_{d}:=\set{\balpha=(\alpha_1,\dots,\alpha_i,\dots,\alpha_n)\in\NN^n_{d}\:\mid\: \alpha_j=0,  \: \forall j\neq i}$. Then, from any sequence $\y\in\RR^{\NN^n_{2d}}$ we can extract the corresponding subsequence $\displaystyle\y_{[i]}\in\RR^{\NN_{2d}^{n{[i]}}}$ representing only the pseudo-moments related to the variable $x_i$. Consequently, ${\bf M}_d(\y_{[i]})$ denotes the \textit{marginal} pseudo-moment matrix of size $(d+1)$, associated to the variable $x_i$. For example, if $n=2, d=1$ and $i=2$, we obtain
\begin{align}
    \y_{[2]}=(1,y_{01},y_{02})\quad\text{and}\quad {\bf M}_1(\y_{[2]})=\begin{bmatrix}
        1& y_{01}\\
        y_{01} & y_{02}
    \end{bmatrix}.
\end{align}
Furthermore, we define the \textit{marginal} Christoffel polynomial associated to the variable $x_i$ to be
\begin{align}
    \widetilde{\Lambda}_d^{\y_{[i]}}:\RR\ni x_i \mapsto\sum_{j=1}^{d+1}\frac{p_j^2(x_i)}{e_j+\beta}\in\RR_+,
\end{align}
where $\beta>0$, and the orthonormal polynomials $(p_j)_{j\leq d+1}$ and eigenvalues $(e_j)_{j\leq d+1}$ are obtained from the spectral decomposition of ${\bf M}_d(\y_{[i]})$.   Analogously, associated $\gamma$-sublevel sets are denoted by $\widetilde{S}_d(\y_{[i]},\gamma)$.
\par
Let us recall that $\y^{*}\in\RR^{\NN^n_{2d}}$ denotes the optimal pseudo-moment sequence of the moment relaxation at order $d$. Then, $\hh$ relies on constructing $n$ sublevel constraints, one for each coordinate $x_i$. These constraints are derived from the sublevel sets of marginal Christoffel polynomials, with the sublevel threshold $\gamma_i>0$ being defined as the evaluation of the marginal Christoffel polynomial $\widetilde{\Lambda}_1^{\y^{*}_{[i]}}$ at the corresponding coordinate of the local solution $\overline{\x} \in \RR^n$, namely
\begin{align}\label{eq: lambda_i}
    \gamma_i := \widetilde{\Lambda}_1^{\y^{*}_{[i]}}(\overline{x}_i), \;i\in\set{1,\dots,n}.
\end{align}
The effectiveness of this approach depends critically on the quality of the local solution. Bad local solutions can cause the Christoffel polynomials to fail in capturing the correct information. We can partially address these risks by carefully selecting the sublevel sets $\widetilde{S}_1(\y_{[i]},\gamma_i)$ to be intersected with the initial feasible set $K$. To do so, we leverage the following property of the Christoffel polynomials:
\begin{prop}\label{prop: minimizer}
Let $\y^{*}\in\RR^{\NN^n_{2d}}$ be such that matrix ${\bf M}_1(\y^{*})\in\mathcal{S}_+^{n+1}$ is invertible. Define $\hat{\x}:=(L_{\y^{*}}(x_1),\dots,L_{\y^{*}}(x_n))\in\RR^n. $ Then, the Christoffel polynomial $\Lambda_1^{\y^{*}}$ verifies
\begin{align}
    \min_{\x\in\RR^n}\Lambda_1^{\y^{*}}(\x)=1=\Lambda_1^{\y^{*}}(\hat{\x}).
\end{align}
\end{prop}
\begin{proof}
Write
\[
{\bf M}_1(\y^*)=
\begin{bmatrix}1&\hat\x^\top\\ \hat\x&{\bf Y}^*\end{bmatrix}
\qquad\text{and}\qquad
{\bf C}:={\bf Y}^*-\hat\x\hat\x^\top.
\]
Since ${\bf M}_1(\y^*)\succ0$, its Schur complement satisfies ${\bf C}\succ0$. The block inverse formula therefore gives, for every $\x\in\RR^n$,
\begin{align}\label{eq: cdk_order2}
\Lambda_1^{\y^*}(\x)
&=\begin{bmatrix}1\\ \x\end{bmatrix}^{\!\top}{\bf M}_1(\y^*)^{-1}
\begin{bmatrix}1\\ \x\end{bmatrix}
 =1+(\x-\hat\x)^\top {\bf C}^{-1}(\x-\hat\x).
\end{align}
Hence $\Lambda_1^{\y^*}(\x)\ge1$, with equality if and only if $\x=\hat\x$, proving the claim.
\end{proof}

\begin{rem}
In order to simplify the proof, Proposition \ref{prop: minimizer} was stated in terms of $\Lambda^{\y^{*}}_1$, but analogous conclusions can be derived for the regularized function $\widetilde{\Lambda}^{\y^{*}}_1$. In addition, the result from Proposition \ref{prop: minimizer} holds true regardless of whether the vector $\y^{*}$ can be represented by a measure.
\end{rem}

\noindent
Furthermore, if there exists some $i \in \{1, \ldots, n\}$ such that the threshold $\gamma_i$ is significantly larger than 1, then, according to Proposition \ref{prop: minimizer}, the $i$-th coordinate of the local solution $\overline{\x}$ must be far from the relaxation-dependent value $L_{\y^*}(x_i)$. This, in turn, suggests either that the local solution $\overline{x}_i$ deviates significantly from $x^{*}_i$, or that the relaxation quality is poor.
\par
Conversely, values of $\gamma_i$ close to one indicate agreement between the local solver and the relaxation on the potential location of the coordinate $i$ of the true minimizer. In such cases, it appears reasonable to search for the true minimizer within the set
\begin{align}\label{eq: marginal_constraint}
    \set{\x \in \RR^n \: \mid \: x_i \in \widetilde{S}_1\left(\y^{*}_{[i]}, \gamma_i\right)}.
\end{align}
Thus, we modify the initial feasible set by intersecting it with the sets described in \eqref{eq: marginal_constraint}, but only when the threshold $\gamma_i$ is smaller than a user-defined \textit{filtering parameter} $\tau > 1$.
The steps of this local-solution-based strengthening approach $\hh$ are formalized in the  Algorithm \ref{alg:loc-sol-based-heuristic}.
\begin{algorithm}[H]
\caption{Implementing $\hh$}
\label{alg:loc-sol-based-heuristic}
\begin{itemize}[label={}]
\item {\bf Input:} Relaxation order $d$, parameter $\tau > 1$.
\begin{enumerate}
\item Solve the moment relaxation of order $d$. Recover its optimal solution \(\y^*\) and a local solution $\overline{\x}$.
\item Construct the marginal Christoffel polynomials $\widetilde{\Lambda}_1^{\y^{*}_{[i]}}$ and compute the thresholds $\gamma_i=\widetilde{\Lambda}_1^{\y^{*}_{[i]}}(\overline{x}_i)$, for $i\in\set{1,\dots,n}$.
\item Modify the feasible set via $\widetilde{K} = K \cap \displaystyle\bigcap_{i:\:\gamma_i\leq\tau}\set{\x\in\RR^n \: \mid \: x_i\in \widetilde{S}_1\left(\y^{*}_{[i]}, \gamma_i\right)}$
, and solve the moment relaxation of order $d$.
\end{enumerate}
\item \textbf{Output}: Tightened bound $\widetilde{f}_{d}\geq f_d$.
\end{itemize}
\end{algorithm}
\noindent
Notice that it is possible, for instance, for the local solution $\overline{x}_i$ and the corresponding pseudo-moment $L_{\y^{*}}(x_i)$ to be very close to each other, while both are simultaneously significantly far from the true solution $x^{*}_i$. In this case, adding the associated sublevel constraint could result in obtaining an invalid bound. However, for a carefully tuned filtering value $\tau$, the percentage of such cases is quite small, as suggested by our experimental results from Section \ref{sec:num_exp}.
\par

\noindent\textbf{Boolean quadratic problems.}
We now consider $\hh$ on $\{-1,1\}^n$, where the marginal
Christoffel constraints admit a simple interpretation.
Let $m_i=L_{\y^*}(x_i)$. Since $x_i^2=1$, the marginal
pseudo-moment matrix is
\[
{\bf M}_1(\y^*_{[i]})=
\begin{bmatrix}1&m_i\\m_i&1\end{bmatrix},
\]
and the corresponding regularized univariate Christoffel polynomial satisfies
\begin{align}\label{eq:H2-Boolean-CD-new}
\widetilde\Lambda_1^{\y^*_{[i]}}(t)
=\frac{2}{1+\beta+m_i t},
\qquad t\in\{-1,1\}.
\end{align}
For $1<\tau<2/(1+\beta)$, let
$c_\tau:=2/\tau-(1+\beta)$.
Then the condition $\gamma_i\le\tau$ in
Algorithm~\ref{alg:loc-sol-based-heuristic} is equivalent to
\begin{align}\label{eq:H2-Boolean-rule-new}
\overline{x}_i m_i>0,\qquad |m_i|\ge c_\tau.
\end{align}
In this case, the marginal Christoffel constraint added by $\hh$
is equivalent, on the Boolean domain, to
$x_i=\operatorname{sign}(m_i)$.
An added constraint may nevertheless exclude global minimizers.
The following Boolean setting makes it possible to examine when
its selection is consistent with an optimal assignment.

\par\noindent
Consider the Ising/QUBO problem
\[
\min_{\x\in\{-1,1\}^n}
\left\{
\sum_{i<j}J_{ij}x_ix_j+\sum_i h_ix_i
\right\},
\]
where $J_{ji}=J_{ij}$ for $j>i$, and let $S_i:=\sum_{j\ne i}|J_{ij}|$.
\begin{h1prop}[Dominated coordinates and $\hh$]
\label{prop:H2-persist-new}
If $|h_i|>S_i$, then every global minimizer satisfies
$x_i=-\operatorname{sign}(h_i)$, and every optimal first-order
pseudo-moment sequence satisfies
\begin{align}\label{eq:H2-sign-bound-new}
\operatorname{sign}(m_i)=-\operatorname{sign}(h_i),\quad
|m_i|\ge\frac{|h_i|-S_i}{|h_i|+S_i}.
\end{align}
Consequently, any marginal constraint added by $\hh$ for
coordinate $i$ preserves every global minimizer.
\end{h1prop}
\noindent
The proof is given in \hyperref[app:H2-persist-new]{Appendix~A.3}.
The hypothesis $|h_i|>S_i$ is the classical coefficient-dominance
condition, which already determines $x_i$ at every global minimizer.
Proposition~\ref{prop:H2-persist-new} shows how this property appears
in the first-order SDP solution: $m_i$ has the same sign and satisfies
the quantitative bound~\eqref{eq:H2-sign-bound-new}. In particular,
\begin{align}\label{eq:H2-R-new}
\overline{x}_i=-\operatorname{sign}(h_i)\quad\text{and}\quad
|h_i|\ge\frac{1+c_\tau}{1-c_\tau}S_i
\quad\Longrightarrow\quad\gamma_i\le\tau.
\end{align}
Thus, the stronger dominance condition explains when $\hh$ adds the
corresponding marginal constraint through its Christoffel threshold.

\noindent
Finally, let us graphically illustrate the behaviour of Algorithm \ref{alg:loc-sol-based-heuristic}
using Example \ref{ex: Leitmotif}.
\begin{figure}[!htbp]
    \centering
    \begin{subfigure}[b]{0.40\textwidth}
        \centering
        \includegraphics[width=\textwidth,height=\textwidth]{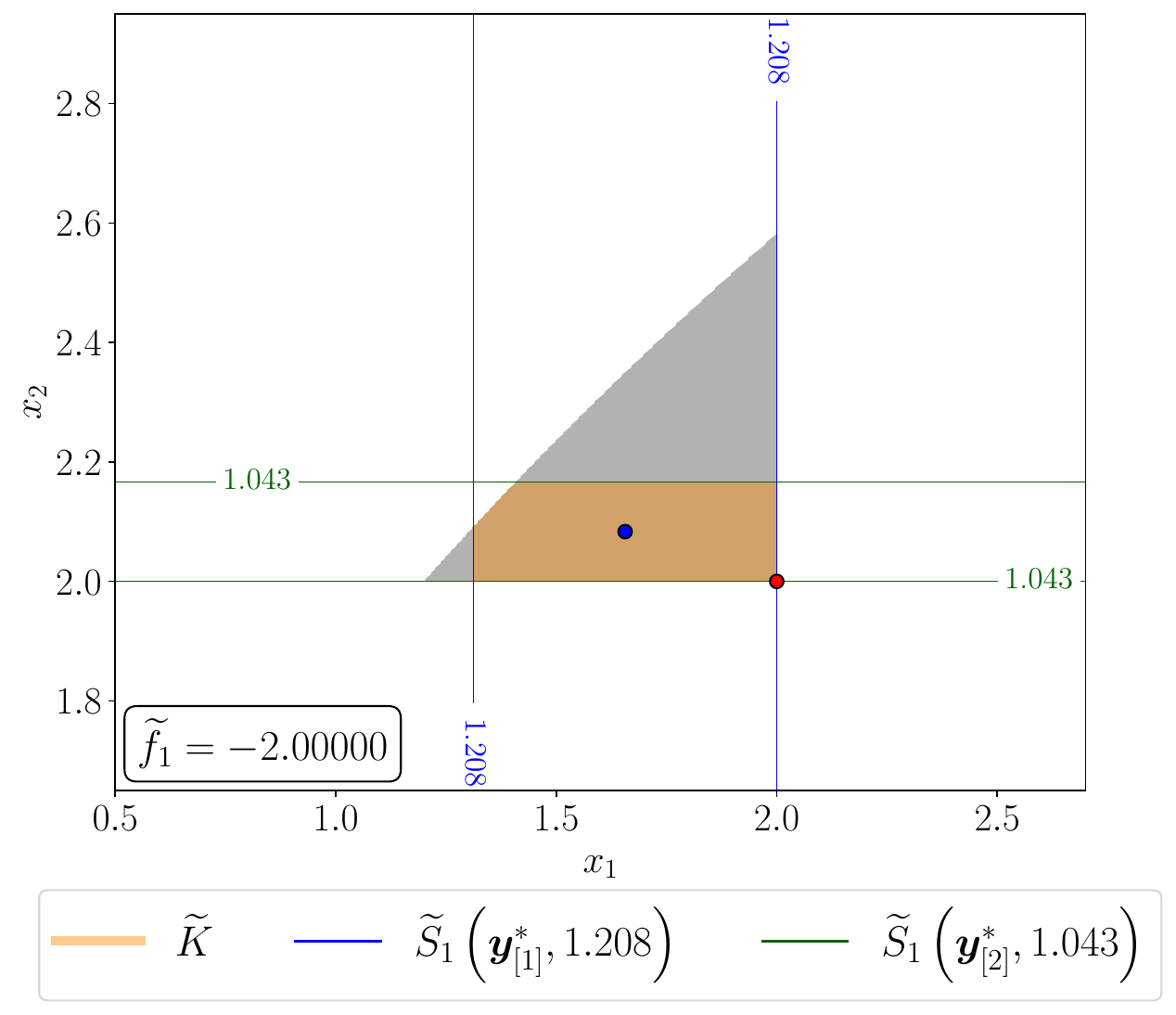}
        \caption{$\tau=1.5$}
    \end{subfigure}
    \hspace{-0.2cm}
    \begin{subfigure}[b]{0.40\textwidth}
        \centering
        \includegraphics[width=\textwidth,height=\textwidth]{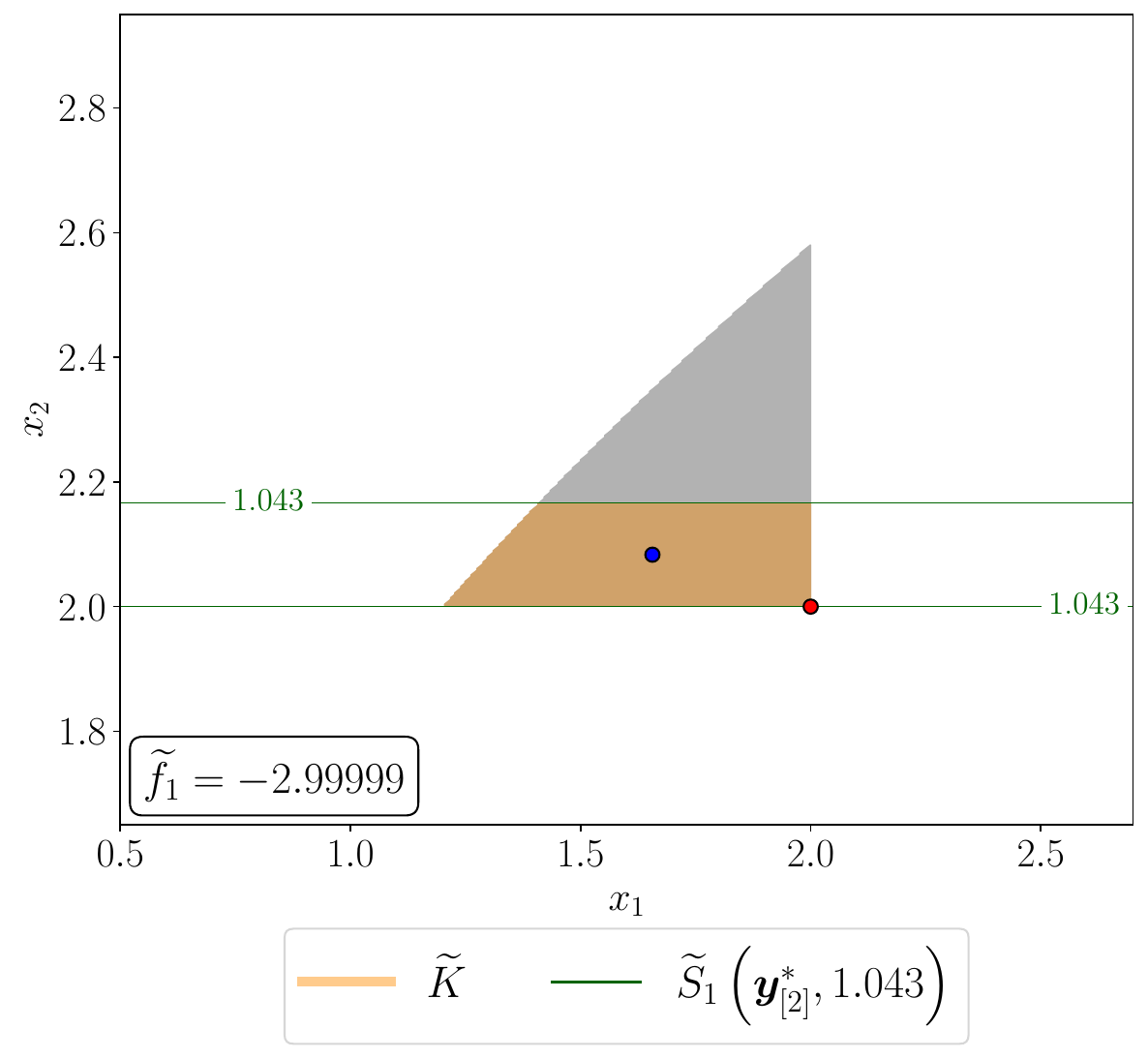}
        \caption{$\tau=1.1$}
    \end{subfigure}
    \caption{Algorithm \ref{alg:loc-sol-based-heuristic} applied to Example \ref{ex: Leitmotif} for different values of $\tau$. The blue point is $\hat{\x}=(1.6562,2.0833)$, and the red point corresponds to $\overline{\x}=\x^{*}=(2,2)$.}
    \label{fig: H2_illustration}
\end{figure}
\FloatBarrier
\noindent
In the case of Example \ref{ex: Leitmotif}, the locally optimal solution to which we had access happened to be globally optimal as well. Moreover, given the previously optimal $\y^{*}$, setting the filtering parameter to $\tau=1.5$, for instance, results in both $\widetilde{S}_1(\y^{*}_{[1]},1.208)$ and $\widetilde{S}_1(\y^{*}_{[2]},1.043)$ being intersected with the initial feasible set, since both $\gamma_1$ and $\gamma_2$ are smaller than $\tau=1.5$. Consequently, the new relaxation bound improves significantly, yielding $\widetilde{f}_1 = -2$. In contrast, setting $\tau = 1.1$ results in only $\widetilde{S}_1(\y^{*}_{[2]}, 1.043)$ being considered, as the certainty regarding the location of $x_1$ is not deemed sufficient. This, in turn, leads to a less restrictive modification of the initial feasible set, resulting in no improvement in the bound, i.e., $\widetilde{f}_1 = -3$.
\par
Finally, an iterative version of $\hh$ can also be formulated, where sublevel sets are updated based on the best available local solution at each iteration. However, unlike $\HH$, if the local solver gets trapped in a poor local minimum, the resulting sublevel sets may fail to exclude suboptimal pseudo-moment sequences, preventing further bound improvements. For completeness, we include the iterative variant, called $\hhi$, in \hyperref[app:iH2]{Appendix~A.4}.

\noindent
The analysis of Ising problems also extends to $\hhi$.
When a valid marginal constraint determines
$x_j=\overline{x}_j$, substituting this value into the
objective replaces each remaining linear coefficient
$h_i$ by $h_i+J_{ij}\overline{x}_j$.
If, at every iteration, at least one remaining variable
satisfies the conditions of Proposition~\ref{prop:H2-persist-new}
that guarantee the addition of a valid constraint,
the procedure reaches a global minimizer after at most
$n$ iterations.

\section{Numerical Experiments}
\label{sec:num_exp}

\noindent
In this section, we present the numerical results of applying our heuristics to different classes of polynomial optimization problems with inexact relaxations of order $d$.
\par
All experiments were performed on a computer with a 13th Gen Intel Core i7-13620H CPU @ 2.40 GHz, 10 cores, 16 logical processors, and 32GB of RAM.
Exact optimal values of all studied optimization problems were computed on one of the \href{https://www.calmip.univ-toulouse.fr/}{CALMIP} servers with 2 Intel Xeon Gold 6140 CPUs @ 2.30 GHz, with a total of 72 logical processors (18 cores per socket, 2 threads per core) and a RAM of 600GB. In most cases, $f_{\min}$ was obtained at the second step of the hierarchy. All the studied POP instances were modeled using the Julia library TSSOS \cite{Magron2021TSSOSAJ}, and their corresponding Moment-SOS relaxations were solved with Mosek \cite{andersen2000mosek}. Locally optimal solutions were obtained using the local solver Ipopt \cite{Wachter2006ipopt}.
The numerical implementation of our algorithms, along with the corresponding experimental results, are available in this GitHub repository: \href{https://github.com/SoDvc2226/CDK_Bound_Strengthening}{\texttt{CDK\_Bound\_Strengthening}.
}

\subsection{QCQP over \texorpdfstring{$[0,1]^n$}{}}\label{exp: dense cases}
The first part of this section focuses on quadratically constrained quadratic programs (QCQPs) of the following form:
\begin{equation}\label{QCQP-classic}
\min_{\x\in\mathbb{R}^{n}} \set{ \x^\top {\bf Q} \x + \x^\top \q \mid \textbf{0}\leq \x\leq\textbf{1} }, \: {\bf Q}\in\mathcal{S}^n, \q\in \RR^n.
\end{equation}
\noindent
The first-order Moment-SOS relaxation of these problems is not always exact, which provides an interesting benchmark for testing the effectiveness of our heuristics. We consider two types of input data $({\bf Q}, \q)$, both drawn from $\mathcal{N}(0,1)$, with either $s=0\%$ or $s=20\%$ sparse entries.
\par
We create $50$ instances of these problems by varying the random seeds, considering two different dimensions: $n = 20$, and $n = 30$.
For the moment, we restrict ourselves to medium-size dimensions to ensure access to the true optimal values of the associated POPs, since those true values are required to assess the efficiency of our methods. In scenarios where problems \textit{lack inherent structure}
---as is the case here---computing exact optimal values in higher dimensions is neither reasonable nor practical. Table \ref{tab:notations} contains explanations of different metrics that were used to assess the performance of our methods.
\begin{table}[!htbp]
    \centering
\begin{tabular}{|c|p{0.7\textwidth}|}
        \hline
        \textbf{Notation} & \textbf{Description} \\ \hline
        $\widetilde{f}_1 > f_{\min}$ & Number of instances (out of $50$) where the method over-restricted the feasible set, leading to an \textit{invalid}  bound. \\ \hline
        Gap (\%) &
        Relative optimality gaps computed before ($\Delta(f_1, f_{\min})$) and after ($\Delta(\widetilde{f}_1, f_{\min}))$ applying the heuristic. \\ \hline
        Time ($\mathrm{s}$) & Computational time (in seconds) required for solving the second-order relaxation ($f_2$) and for computing the heuristic bounds ($\widetilde{f}_1$). \\ \hline
        Solved & Number of instances where
        $\Delta(\widetilde{f}_1, f_{\min})$ is below the tolerance $\delta \%$. \\ \hline
        $k$ & Number of iterations. \\ \hline
    \end{tabular}
    \caption{Summary of column names used in Tables \ref{tab:H1-Qcombined}, \ref{tab:H2-Qcombined}, \ref{tab: H1 cs}, \ref{tab:H2 - Q cs} and the iH2 table in \hyperref[app:iH2]{Appendix~A.4}. Columns ``Gap (\%)'', ``Time ($\mathrm{s}$)" and ``$k$" represent \textit{averages} over the instances for which a \textit{valid} post-heuristic bound $\widetilde{f}_1$ was obtained.}
    \label{tab:notations}
\end{table}
\subsubsection{Performance of \texorpdfstring{$\HH$}{H1}}
We start by assessing the performance of $\HH$, implemented via Algorithm \ref{alg:iteration-based-heuristic}, where we set optimality gap tolerance to be $\delta=0.5\%$, and a maximum number of iterations $N=15$. The kernel dimension was determined by counting the number of eigenvalues of the corresponding moment matrices smaller than $10^{-3}$. The regularization parameter for computing the Christoffel sublevel sets was set to $\beta = 10^{-5}$.
\begin{table}[!htbp]
    \centering \small
    \renewcommand{\arraystretch}{1.20}
    \begin{tabular}{|c|c|c|c|c|c|c|c|c|c|}
        \hline
        \textbf{$n$} & \textbf{$s$} & \textbf{$\varepsilon$} & \textbf{$\widetilde{f}_1 > f_{\min}$} & \multicolumn{2}{c|}{Gap (\%)} & \multicolumn{2}{c|}{Time $(\mathrm{s})$} & Solved & \textbf{$k$} \\
        \cline{5-8}
         &  &  &  & $\Delta(f_1, f_{\min})$ & $\Delta(\widetilde{f}_1, f_{\min})$ & $\widetilde{f}_1$ & $f_2$ &  &  \\
        \hline\addlinespace\hline
        \multirow{8}{*}{\rotatebox{90}{20}}
        & \multirow{4}{*}{\rotatebox{90}{$0\:\%$}}
                & $0.01$ & \cellcolor{gray!15}1 & \cellcolor{gray!15}4.853 & \cellcolor{gray!15}2.755 & \cellcolor{gray!15}0.397 & \cellcolor{gray!15}26.256 & \cellcolor{gray!15}13 & \cellcolor{gray!15}13.80 \\ \cline{3-10}
        &  & $0.05$ & \cellcolor{gray!15}4 & \cellcolor{gray!15}4.649 & \cellcolor{gray!15}1.687 & \cellcolor{gray!15}0.253 & \cellcolor{gray!15}26.387 & \cellcolor{gray!15}20 & \cellcolor{gray!15}8.09 \\ \cline{3-10}
        &  & $0.1$  & \cellcolor{gray!15}2 & \cellcolor{gray!15}4.757 & \cellcolor{gray!15}1.629 & \cellcolor{gray!15}0.221 & \cellcolor{gray!15}26.404 & \cellcolor{gray!15}16 & \cellcolor{gray!15}6.65 \\ \cline{3-10}
        &  & $0.15$ & \cellcolor{gray!15}2 & \cellcolor{gray!15}4.757 & \cellcolor{gray!15}2.132 & \cellcolor{gray!15}0.161 & \cellcolor{gray!15}26.404 & \cellcolor{gray!15}11 & \cellcolor{gray!15}5.02 \\ \cline{2-10}
               & \multirow{4}{*}{\rotatebox{90}{$20\:\%$}}
                & $0.01$ & 2 & 4.668 & 2.591 & 0.407 & 26.659 & 13 & 14.10 \\ \cline{3-10}
        &  & $0.05$ & 3 & 4.540 & 1.256 & 0.259 & 26.727 & 22 & 8.43 \\ \cline{3-10}
        &  & $0.1$  & 4 & 4.415 & 1.379 & 0.195 & 26.739 & 20 & 6.37 \\ \cline{3-10}
        &  & $0.15$ & 4 & 4.304 & 1.732 & 0.136 & 26.733 & 14 & 4.28 \\
\hline\addlinespace\hline
        \multirow{8}{*}{\rotatebox{90}{30}}
        & \multirow{4}{*}{\rotatebox{90}{$0\:\%$}}
                & $0.01$ & \cellcolor{gray!15}0 & \cellcolor{gray!15}5.419 & \cellcolor{gray!15}3.962 & \cellcolor{gray!15}0.869 & \cellcolor{gray!15}1206.587 & \cellcolor{gray!15}9 & \cellcolor{gray!15}14.00 \\ \cline{3-10}
        &  & $0.05$ & \cellcolor{gray!15}4 & \cellcolor{gray!15}5.318 & \cellcolor{gray!15}2.361 & \cellcolor{gray!15}0.728 & \cellcolor{gray!15}1211.969 & \cellcolor{gray!15}16 & \cellcolor{gray!15}10.98 \\ \cline{3-10}
        &  & $0.1$  & \cellcolor{gray!15}5 & \cellcolor{gray!15}5.404 & \cellcolor{gray!15}2.004 & \cellcolor{gray!15}1.007 & \cellcolor{gray!15}1209.611 & \cellcolor{gray!15}13 & \cellcolor{gray!15}10.16 \\ \cline{3-10}
        &  & $0.15$ & \cellcolor{gray!15}6 & \cellcolor{gray!15}5.132 & \cellcolor{gray!15}1.749 & \cellcolor{gray!15}0.829 & \cellcolor{gray!15}1208.643 & \cellcolor{gray!15}14 & \cellcolor{gray!15}7.36 \\ \cline{2-10}
        & \multirow{4}{*}{\rotatebox{90}{$20\:\%$}}
               & $0.01$ & 0 & 6.059 & 4.784 & 0.899 & 1163.134 & 2 & 14.84 \\ \cline{3-10}
        &  & $0.05$ & 2 & 5.959 & 2.927 & 0.870 & 1165.358 & 12 & 13.04 \\ \cline{3-10}
        &  & $0.1$  & 2 & 5.871 & 2.300 & 0.819 & 1165.310 & 13 & 11.94 \\ \cline{3-10}
        &  & $0.15$ & 9 & 5.074 & 1.790 & 0.566 & 1172.107 & 10 & 8.41 \\ \hline
    \end{tabular}
    \caption{Performance of $\HH$ in dimensions $20$ and $30$, for different values of the penalization factor $\varepsilon$.}
    \label{tab:H1-Qcombined}
\end{table}
\noindent
From Table \ref{tab:H1-Qcombined}, we observe that when $\varepsilon $ is small (e.g., $\varepsilon =0.01$), the method progresses cautiously, requiring on average a larger number of iterations, regardless of the dimension. On the other hand, moderate values of $\varepsilon $ (e.g., $\varepsilon = 0.05$ or $\varepsilon = 0.1$) significantly improve optimality gaps while inducing smaller running times. Finally, the feasible set restrictions imposed when $\varepsilon = 0.15$ appear to be overly aggressive, as evidenced by the fact that the number of \textit{Solved} cases is never maximized for this value.
\par\noindent
Regardless of the choice of $\varepsilon$ and the dimension $n$, the average post-heuristic gaps decrease significantly, while the number of cases yielding an \text{invalid} bound remains relatively small. Moreover, note that random instances from Table \ref{tab:H1-Qcombined} with $20\%$ zero entries appear to be harder to solve compared to the dense ones, as indicated by their relatively higher pre-heuristic gaps. Nonetheless, $\HH$ successfully reduces these gaps by at least half in seven out of eight parameter configurations.
\par
Furthermore, we observe that moderate choices of the penalization factor, such as $\varepsilon = 0.05$ or $\varepsilon = 0.1$, typically solve the highest number of cases to optimality much before reaching the maximum number of iterations. In other words, for these values of $\varepsilon$, Algorithm \ref{alg:iteration-based-heuristic} appears to be able to certify that further reductions of the feasible set are unnecessary, as they would either yield invalid bounds or fail to significantly reduce the gap.
\par
Finally, we highlight the efficiency of $\HH$ compared to the computationally expensive second-order relaxation. For example, when $n = 30$ and $s=20\%$, recovering the true optimal values using the second-order relaxation takes an average of $1163.134 \:\mathrm{s}$. In contrast, Algorithm \ref{alg:iteration-based-heuristic} with $\varepsilon = 0.1$ takes only $0.819 \:\mathrm{s}$, solving 13 cases and reducing the average gap by half.
Additionally, it is worth emphasizing that, despite requiring multiple iterations,  $\HH$ remains much faster than solving the second-order relaxation.
\par\smallskip\noindent
The following figure illustrates how the choice of the penalization factor $\varepsilon$ affects the rate of improvement of the relaxation bound, under two scenarios differing by the quality of the best available upper bound $\operatorname{ub}$.
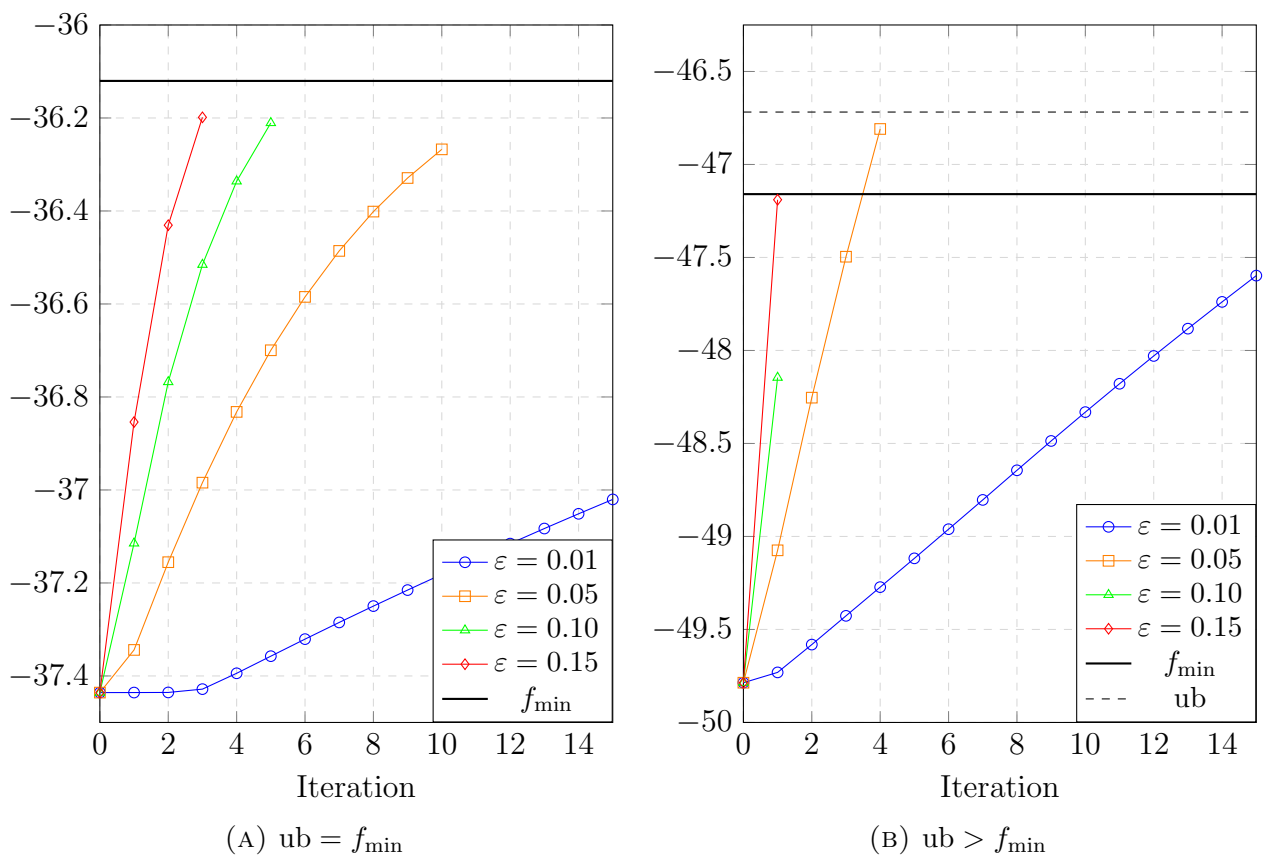
\begin{figure}[!htbp]
    \centering
\begin{subfigure}{0.49\textwidth}
    \begin{tikzpicture}
    \begin{axis}[
        width=\textwidth,
        height=1.0\textwidth,
        xlabel={Iteration},
        legend style={at={(1,0)}, anchor=south east, font=\small},
        grid=major,
        grid style={dashed,gray!30},
        ymin=-37.5, ymax=-36,
        xmin=0, xmax=15,
    ]
    \addplot[mark=o, color=blue] coordinates {
        (0, -37.435772) (1, -37.43569) (2, -37.435455) (3, -37.42834)
        (4, -37.394283) (5, -37.357459) (6, -37.320785) (7, -37.284896)
        (8, -37.249754) (9, -37.215301) (10, -37.181475) (11, -37.148204)
        (12, -37.115368) (13, -37.082949) (14, -37.051269) (15, -37.020327)
    };
    \addlegendentry{$\varepsilon=0.01$}

    \addplot[mark=square, color=orange] coordinates {
        (0, -37.435772) (1, -37.344035) (2, -37.155315) (3, -36.984296)
        (4, -36.832113) (5, -36.699565) (6, -36.584808) (7, -36.485886)
        (8, -36.4013) (9, -36.329079) (10, -36.267473)
    };
    \addlegendentry{$\varepsilon=0.05$}

    \addplot[mark=triangle, color=green] coordinates {
        (0, -37.435772) (1, -37.114725) (2, -36.767737) (3, -36.515388)
        (4, -36.336113) (5, -36.210619)
    };
    \addlegendentry{$\varepsilon=0.10$}

    \addplot[mark=diamond, color=red] coordinates {
        (0, -37.435772) (1, -36.853915) (2, -36.430542) (3, -36.198609)
    };
    \addlegendentry{$\varepsilon=0.15$}

    \addplot[black, thick] coordinates {(0, -36.12029578315779) (15, -36.12029578315779)};
    \addlegendentry{$f_{\mathrm{min}}$}
    \end{axis}
    \end{tikzpicture}
    \caption{$\operatorname{ub}=f_{\mathrm{min}}$}
\end{subfigure}
\begin{subfigure}{0.49\textwidth}
    \begin{tikzpicture}
    \begin{axis}[
        width=\textwidth,
        height=1.0\textwidth,
        xlabel={Iteration},
        legend style={at={(1,0)}, anchor=south east, font=\small},
        grid=major,
        grid style={dashed,gray!30},
        ymin=-50, ymax=-46.25,
        xmin=0, xmax=15,
    ]
    \addplot[mark=o, color=blue] coordinates {
        (0, -49.786399) (1, -49.730886) (2, -49.581332) (3, -49.426782)
        (4, -49.272441) (5, -49.117707) (6, -48.960973) (7, -48.803487)
        (8, -48.64427) (9, -48.486539) (10, -48.331449) (11, -48.179045)
        (12, -48.02952) (13, -47.882813) (14, -47.738865) (15, -47.597578)
    };
    \addlegendentry{$\varepsilon=0.01$}

    \addplot[mark=square, color=orange] coordinates {
        (0, -49.786399) (1, -49.074489) (2, -48.25373)
        (3, -47.495936) (4, -46.809186)
    };
    \addlegendentry{$\varepsilon=0.05$}

    \addplot[mark=triangle, color=green] coordinates {
        (0, -49.786399) (1, -48.145749)
    };
    \addlegendentry{$\varepsilon=0.10$}

    \addplot[mark=diamond, color=red] coordinates {
        (0, -49.786399) (1, -47.189737)
    };
    \addlegendentry{$\varepsilon=0.15$}

    \addplot[black, thick] coordinates {(0, -47.159442988247534) (15, -47.159442988247534)};
    \addlegendentry{$f_{\mathrm{min}}$}

    \addplot[black, dashed] coordinates {(0, -46.71817328139224) (15, -46.71817328139224)};
    \addlegendentry{$\operatorname{ub}$}
    \end{axis}
    \end{tikzpicture}
    \caption{$\operatorname{ub}>f_{\min}$}
\end{subfigure}
\caption{Trajectories~$(\widetilde{f}_{1,k})_{0 \leq k \leq 15}$ from Algorithm~\ref{alg:iteration-based-heuristic} for different values of $\varepsilon$, in dimension $n=20$, with  $(\mathbf{Q},\q)$ dense and randomly generated, using seeds 16 (left) and 38 (right).} \label{fig: H1_illustration}
\end{figure}

\noindent
The left subfigure of Figure \ref{fig: H1_illustration} demonstrates that a valid lower bound was obtained for all tested values of $\varepsilon$. It also reveals that smaller values of $\varepsilon$ reduce the speed at which the bound improves. In contrast, the right subfigure of Figure \ref{fig: H1_illustration} shows a particular instance where an invalid bound was obtained for $\varepsilon=0.05$. Unlike the left subfigure, Algorithm \ref{alg:iteration-based-heuristic} failed to recover a tight upper bound $\operatorname{ub}$ in this case, leading to inaccurate measurements of the relative optimality gap.
Despite this issue, for more severe penalization factors, such as  $\varepsilon=0.15$ or $\varepsilon=0.1$,  $\HH$ was able to certify that for any $k \geq 2$, we must have $\widetilde{f}_{1,k} > f_{\min}$. As a result, the best proposed bound $\widetilde{f}_{1,1}$ satisfied $f_1 < \widetilde{f}_{1,1} < f_{\min}$. Moreover, cautious reductions of the feasible set generally avoid producing invalid bounds, though at the expense of requiring a higher number of iterations. This observation holds regardless of the availability of tight upper bounds.

\FloatBarrier
\subsubsection{Performance of \texorpdfstring{$\hh$}{H2}}
We continue by assessing the performance of $\hh$, implemented via Algorithm \ref{alg:loc-sol-based-heuristic}. The kernel dimension of the marginal moment matrices was determined by counting the number of eigenvalues smaller than $10^{-3}$. The regularization parameter for computing the marginal Christoffel sublevel sets was set to $\beta = 10^{-3}$.
\begin{table}[!htbp]
\centering \small
\renewcommand{\arraystretch}{1.18}
\begin{tabular}{|c|c|c|c|c|c|c|c|c|}
    \hline
    \multirow{2}{*}{$n$} & \multirow{2}{*}{$s$} & \multirow{2}{*}{$\tau$} & \multirow{2}{*}{$\widetilde{f}_1>f_{\min}$} & \multicolumn{2}{c|}{Gap (\%)} & \multicolumn{2}{c|}{Time ($\mathrm{s}$)} & \multirow{2}{*}{Solved} \\
    \cline{5-6} \cline{7-8}
    & & & & $\Delta(f_1,f_{\min})$ & $\Delta(\widetilde{f}_1,f_{\min})$ & $\widetilde{f}_1$ & $f_2$ & \\
    \hline
    \addlinespace
    \hline
    \multirow{8}{*}{\rotatebox{90}{$20$}}
        & \multirow{4}{*}{\rotatebox{90}{$0\:\%$}}
        & None  & \cellcolor{gray!15}4 & \cellcolor{gray!15}4.834 & \cellcolor{gray!15}0.410 & \cellcolor{gray!15}0.083 & \cellcolor{gray!15}26.339 & \cellcolor{gray!15}39 \\
    \cline{3-9}
        &  & $1.5$  & \cellcolor{gray!15}4 & \cellcolor{gray!15}4.834 & \cellcolor{gray!15}0.588 & \cellcolor{gray!15}0.086 & \cellcolor{gray!15}26.339 & \cellcolor{gray!15}35 \\
    \cline{3-9}
        &  & $1.25$ & \cellcolor{gray!15}1 & \cellcolor{gray!15}4.933 & \cellcolor{gray!15}1.212 & \cellcolor{gray!15}0.085 & \cellcolor{gray!15}26.311 & \cellcolor{gray!15}28 \\
    \cline{3-9}
        &  & $1.1$  & \cellcolor{gray!15}0 & \cellcolor{gray!15}4.934 & \cellcolor{gray!15}2.314 & \cellcolor{gray!15}0.094 & \cellcolor{gray!15}26.268 & \cellcolor{gray!15}19 \\
    \cline{2-9}
        & \multirow{4}{*}{\rotatebox{90}{$20\:\%$}}
        & None  & 7 & 4.317 & 0.268 & 0.026 & 26.793 & 38 \\
    \cline{3-9}
        &  & $1.5$  & 5 & 4.395 & 0.391 & 0.045 & 26.734 & 36 \\
    \cline{3-9}
        &  & $1.25$ & 1 & 4.740 & 0.878 & 0.044 & 26.560 & 29 \\
    \cline{3-9}
        &  & $1.1$  & 1 & 4.740 & 2.177 & 0.045 & 26.560 & 17 \\
    \hline
    \addlinespace
    \hline
    \multirow{8}{*}{\rotatebox{90}{$30$}}
        & \multirow{4}{*}{\rotatebox{90}{$0\:\%$}}
        & None  & \cellcolor{gray!15}11 & \cellcolor{gray!15}5.092 & \cellcolor{gray!15}0.525 & \cellcolor{gray!15}0.304 & \cellcolor{gray!15}1209.983 & \cellcolor{gray!15}32 \\
    \cline{3-9}
        &  & $1.5$  & \cellcolor{gray!15}6 & \cellcolor{gray!15}5.453 & \cellcolor{gray!15}0.757 & \cellcolor{gray!15}0.296 & \cellcolor{gray!15}1202.457 & \cellcolor{gray!15}27 \\
    \cline{3-9}
        &  & $1.25$ & \cellcolor{gray!15}4 & \cellcolor{gray!15}5.566 & \cellcolor{gray!15}1.484 & \cellcolor{gray!15}0.306 & \cellcolor{gray!15}1203.910 & \cellcolor{gray!15}18 \\
    \cline{3-9}
        &  & $1.1$  & \cellcolor{gray!15}2 & \cellcolor{gray!15}5.527 & \cellcolor{gray!15}2.781 & \cellcolor{gray!15}0.323 & \cellcolor{gray!15}1207.392 & \cellcolor{gray!15}11 \\
    \cline{2-9}
                & \multirow{4}{*}{\rotatebox{90}{$20\:\%$}}
        & None  & 5 & 6.059 & 0.740 & 0.094 & 1172.473 & 30 \\
    \cline{3-9}
        &  & $1.5$  & 2 & 6.030 & 1.022 & 0.087 & 1167.344 & 29 \\
    \cline{3-9}
        &  & $1.25$ & 0 & 6.059 & 1.780 & 0.093 & 1163.134 & 12 \\
    \cline{3-9}
        &  & $1.1$  & 0 & 6.059 & 3.364 & 0.092 & 1163.134 & 5 \\
    \hline
\end{tabular}
\caption{Performance of $\hh$ in dimensions $20$ and $30$, for different values of the filtering parameter $\tau$ (``None'' corresponds to cases where no filtering is used).}
\label{tab:H2-Qcombined}
\end{table}
\noindent
Firstly, notice that the runtime of Algorithm \ref{alg:loc-sol-based-heuristic} is usually smaller compared to the running time of Algorithm \ref{alg:iteration-based-heuristic} over the same POP instances. This is due to the fact that Algorithm \ref{alg:loc-sol-based-heuristic} requires only one iteration.
However, Algorithm \ref{alg:loc-sol-based-heuristic} is highly dependent on the quality of the available local solutions. For example,  when $n=30$ and $s=0\%$, 15 out of 50 available local solutions yielded strict upper bounds, and Algorithm \ref{alg:loc-sol-based-heuristic}, without enabled filtering, provided invalid bounds in 11 out of those 15 cases.
\par\noindent
The number of those undesirable scenarios can be reduced by cautiously filtering the sublevel constraints, which happens, for instance, when $\tau=1.1$. However, this conservative approach yields significantly smaller average gap improvement, illustrating the trade-off between the risk of over-restriction and the possibility of highly significant bound improvements. \\
\par\smallskip
\noindent To summarize, both $\HH$ and $\hh$ effectively strengthen the bounds of moment relaxations while being significantly less computationally prohibitive than higher-order relaxations. The reductions of the feasible set induced by $\HH$ are gradual and thus more cautious, particularly for small values of $\varepsilon$. Due to the weaker dependence of $\HH$ on the quality of available upper bounds, we conjecture that $\HH$ would be preferable for general classes of more challenging problems, where obtaining high-quality locally optimal solutions is difficult. Conversely, for specific classes of POP problems where high-quality local solutions are readily available, strengthening the bounds via $\hh$ would be a highly reasonable approach.

\subsection{Exploiting correlative sparsity}\label{exp: sparse case}
So far, we have only considered POP instances that present no underlying structure that could be leveraged to increase the efficiency of the associate SDP solving procedure. However, many optimization problems possess important structural properties, such as sparsity. In this subsection, we focus on \textit{correlative sparsity} \cite{waki2006sums,LasserreCSP2006} and demonstrate the efficiency of our proposed approaches within this framework.
\par
Before presenting our experimental results, let us provide a short reminder of the correlative-sparsity-based Moment-SOS hierarchy.
\par
Let us suppose that $\set{1,\dots,n}=:I_0=\cup_{l=1}^p I_l$ with $I_l$ not necessarily disjoint and let $n_l= | I_l | $. To each subset $I_l$,
we can associate the corresponding subset of decision variables $\x_{I_l}:=\{x_i, i\in I_l\}$. We say that an instance of POP exhibits correlative sparsity if
\begin{itemize}
    \item There exist $(f_l)_{l\in\set{1,\dots,p}}$ such that $\displaystyle f=\sum_{l=1}^p f_l$, with $f_l\in\mathbb{R}[\bm{x}_{I_l}]$,
    \item  The polynomials $(g_j)_{j\in\set{1,\dots,m}}$ can be split into disjoints sets $J_l$, such that $g_j\in J_l$ if and only if  $g_j\in\mathbb{R}[\bm{x}_{I_l}]$.
\end{itemize}
Then, the \textit{correlatively sparse moment relaxatio}n of order $d\geq d_{\min}$ is given by:
\begin{align}
    \label{eq:prima_cs}
    {\bf P}_d^{\operatorname{cs}}: \quad
    \begin{cases}
          \displaystyle f_d^{\operatorname{cs}}:=\inf_{\y_l \in \RR^{\NN^{n_l}_{2d}}}\quad & \displaystyle\sum_{l=1}^pL_{\y_l}(f_l)\\
        \text{s.t.} \quad &{\bf M}_d(\y_l) \in \mathcal{S}_+^{s(n_l, d)}, \quad l\in\set{1,\dots,p},\\
        &  {\bf M}_{d-d_j}(g_j\y_l) \in \mathcal{S}_+^{s(n_l,d-d_j)}, \quad j\in J_l, \quad l\in\{1,\dotsc,p\},\\
        & L_{y_l}(1)=1, \quad l\in\{1,\dotsc,p\},
    \end{cases}
\end{align}
where we let $\y_l :=\set{y_{\balpha} \: \mid \balpha\in\NN^n_{2d}, \: \forall i\in\set{1,\dots,s(n,d)}, \: \alpha_i=0, \: \forall i\notin I_l}$.
The corresponding SOS relaxation can be built analogously. A detailed convergence analysis of the hierarchy from \eqref{eq:prima_cs} is available in \cite{LasserreCSP2006}.
\par
Notice that, if $\displaystyle\rho:=\max_{l\in\set{1,\dots,p}} n_l$, then SDP relaxations from \eqref{eq:prima_cs} contain matrices of maximum size $s(\rho,d)=\mathcal{O}(\rho^{d})$, yielding a significant improvement when $\rho \ll n$. However, given a relaxation order $d\geq d_{\min}$, the bound $f^{\cs}_d$ can be arbitrarily far from $f_{\min}$.
In high-dimensional settings, where even the size of $\rho$ is not negligible, improving the bound by ascending the hierarchy (e.g., computing $f^{\cs}_{d+1}$, $f^{\cs}_{d+2}$, and so on) remains a significant computational challenge. Nonetheless, we demonstrate that our proposed approaches are capable of providing strengthened bounds $\widetilde{f}^{\cs}_d$ within this framework.
The clique-wise implementations of $\HH$ and $\hh$, denoted
$\HHCS$ and $\hhcs$, are given in
\hyperref[app:sparsity]{Appendix~A.5}.
\noindent

Let us consider the following optimization problem:
\begin{align}\label{QCQP: sparse}
    \min_{\x\in\RR^n}\set{\x^\top \q + \sum_{l=1}^p \x_{I_l}^\top{\bf Q}_l\x_{I_l} \mid \textbf{0}\leq \x\leq\textbf{1} },
\end{align}
where ${\bf Q}_l\in\mathcal{S}^{n_l}$ for each $l\in\set{1,\dots,p}$, $\q\in \RR^n$, and $\displaystyle n = \sum_{l=1}^{p} n_l-\sum_{l=1}^{p-1} \abs{I_l\cap I_{l+1}}$. We suppose that $\abs{I_l \cap I_{l+1}}>0$ for all $l\in\set{1,\dots,p-1}$. \\
\noindent
 Notice that POP instances of the form \eqref{QCQP: sparse} possess a correlatively sparse structure by construction. These instances can also be written in the form \eqref{QCQP-classic}, where the matrix ${\bf Q} \in \mathcal{S}^n$ is block-diagonal, having $p$ blocks, with each block associated with a distinct subset $I_l$.
 Moreover, solving optimization problems of this form naturally arises in many applications, such as portfolio optimization \cite{Gondzio2007} or classification via support vector machines \cite{SVMQCQP2015}.
\par\smallskip
We generate 50 such POP instances using the following structure: $p=7$, $n_l=22$ for $l\in\set{1,\dots,p}$, and  $\abs{I_l \cap I_{l+1}}=4$ for $l\in\set{1,\dots,p-1}$. Consequently, we get $n=130$. We set $N=15$, $\delta=0.5\%$ and $\beta=10^{-4}$. Results are presented in Table \ref{tab: H1 cs} and  Table \ref{tab:H2 - Q cs}.

\begin{table}[!htbp]
    \centering \small
    \renewcommand{\arraystretch}{1.18}
    \begin{tabular}{|c|c|c|c|c|c|c|c|c|}
        \hline
        \multirow{2}{*}{$n$} & \multirow{2}{*}{$\varepsilon$} & \multirow{2}{*}{$\widetilde{f}^{\cs}_1 > f_{\min}$} & \multicolumn{2}{c|}{Gap (\%)} & \multicolumn{2}{c|}{Time $(\mathrm{s})$} & \multirow{2}{*}{Solved} & \multirow{2}{*}{$k$} \\
\cline{4-7}
         &  &  & $\Delta(f_1^{\cs}, f_{\min})$ & $\Delta(\widetilde{f}^{\cs}_1, f_{\min})$ & $\widetilde{f}^{\cs}_1$ & $f_2^{\cs}$ & & \\
\hline
    \addlinespace
    \hline
        \multirow{5}{*}{$130$}
        &$0.01$ & 0 & 5.890 & 4.670 & 2.939 & 468.294 & 0 & 15.00 \\ \cline{2-9}
        & $0.05$ & 10 & 5.757 & 1.359 & 2.706 & 469.012 & 15 & 13.12 \\ \cline{2-9}
        & $0.1$  & 13 & 5.764 & 0.551 & 1.650 & 466.427 & 21 & 8.35 \\ \cline{2-9}
        & $0.15$ & 11 & 5.661 & 1.073 & 1.080 & 468.437 & 12 & 5.15 \\ \cline{2-9}
        & $0.25$ & 8 & 5.851 & 1.808 & 0.952 & 469.787 & 8 & 2.69 \\ \hline
    \end{tabular}
    \caption{Performance of $\HHCS$ for different values of the penalization factor $\varepsilon$, and randomly generated input data. The number of cases where $\HHCS$ has access only to a strict upper bound is $41$ out of $50$. }
    \label{tab: H1 cs}
\end{table}
\noindent
Firstly, notice that, despite correlative sparsity exploitation, second-order moment relaxation remains costly with respect to $\HHCS$. Moreover, $\HHCS$ effectively reduces the optimality gaps in all considered configurations. Given the high-dimensionality of the problem, a more severe penalization, like $\varepsilon=0.25$, also seems efficient in reducing the gaps while maintaining the number of wrong cases relatively small.

\begin{table}[!htbp]
\small
\centering\renewcommand{\arraystretch}{1.18}
\begin{tabular}{|c|c|c|c|c|c|c|c|}
    \hline
    \multirow{2}{*}{$n$} & \multirow{2}{*}{$\tau$} & \multirow{2}{*}{$\widetilde{f}^{\cs}_1>f_{\min}$} & \multicolumn{2}{c|}{Gap (\%)} & \multicolumn{2}{c|}{Time ($\mathrm{s}$)} & \multirow{2}{*}{Solved} \\
    \cline{4-5} \cline{6-7}
    & & & $\Delta(f_1,f_{\min})$ & $\Delta(\widetilde{f}^{\cs}_1,f_{\min})$ & $\widetilde{f}^{\cs}_1$ & $f_2^{\cs}$ & \\
\hline
    \addlinespace
    \hline
    \multirow{4}{*}{\( 130 \)}
        & None  & 5 & 5.955 & 0.625 & 0.252 & 468.635 & 24 \\
\cline{2-8}
& $1.5$  & 0 & 5.890 & 0.873 & 0.208 & 468.294 & 11 \\
\cline{2-8}
& $1.25$ & 0 & 5.890 & 1.814 & 0.217 & 468.294 & 2 \\
\cline{2-8}
& $1.1$  & 0 & 5.890 & 3.387 & 0.216 & 468.294 & 0 \\
   \hline
\end{tabular}
\caption{Performance of $\hhcs$ for different values of the filtering parameter $\tau$,
and randomly generated input data. The available local solution $\overline{\x}$ is not global, i.e., $f(\overline{\x}) > f_{\min}$, in $41$ out of $50$ cases.}
\label{tab:H2 - Q cs}
\end{table}
\noindent
As demonstrated in Table \ref{tab:H2 - Q cs}, $\hhcs$ appears to be capable of efficiently taking advantage of the information provided by available locally optimal solutions. Furthermore, we observe similar trade-offs as those observed in the absence of correlative sparsity, see Table \ref{tab:H2-Qcombined}. Significantly reduced gaps and a low number of unsuccessful cases further highlight an important observation: the locally optimal solutions provided by Ipopt are generally quite close to the globally optimal solutions. Generally speaking, since the first-order relaxation in these cases is not extremely loose, $\hhcs$ can be understood as assessing the quality of given locally optimal solutions.

\subsection{Beyond QCQPs}

Optimization problems analyzed earlier are quadratic and addressed using first-order (sparse) Moment-SOS relaxation, with our approaches effectively strengthening bounds. However, many POPs involve higher-degree polynomials (e.g., cubic or quartic), necessitating higher-order relaxations. We illustrate the integration of our approaches in such relaxations with the following example.
\par
\noindent
Let ${\bf Q}\in\mathcal{S}^n, \q\in\RR^n$, $(r_1,r_2)\in\RR_+\times\RR_+$, $({\bf c}_1,{\bf c}_2)\in\RR^n\times\RR^n$. Let $g_1:\x\mapsto r_1^2-\|\x-{\bf c}_1\|^2_2$ and $g_2:\x\mapsto r^2_2-\|\x-{\bf c}_2\|^2_2$. Consider the following optimization problem:
\begin{equation}\label{unionballs}
\min_{\x\in K} f(\x):= \x^\top {\bf Q} \x + \x^\top \q ,  \quad K=\set{\x\in\RR^n \: \mid \: -g_1(\x)g_2(\x)\geq0}.
\end{equation}
\noindent
Optimization problems of this form consist of minimizing a quadratic (nonconvex) function over a nonconvex set corresponding to the union of two Euclidean balls. This is a quartic optimization problem, so the minimal order relaxation is the relaxation of order two.
\par
\noindent
We consider the following configuration: $n=5, (r_1,r_2)=(1.0,\sqrt{0.1}), ({\bf c}_1,{\bf c}_2)=({\bf 0},{\bf 1})$, and
    $${\bf Q}=
\begin{bmatrix}
-1.4396 & -0.2259 &  0.0983 & -0.0085 & -2.3838 \\
-0.2259 &  0.8043 &  0.3730 &  1.2719 &  0.1370 \\
 0.0983 &  0.3730 & -1.0236 &  0.0597 &  0.5024 \\
-0.0085 &  1.2719 &  0.0597 &  0.9421 &  1.2085 \\
-2.3838 &  0.1370 &  0.5024 &  1.2085 &  0.7885
\end{bmatrix}, \quad \q = \begin{bmatrix}
    -1.2691\\
-2.9877\\
 2.5349\\
-0.4151\\
 0.1464
\end{bmatrix}.
$$
For this particular POP, we have $f_{\min}=f_3=-5.7161$, and $f_2=-7.3367$, so that $\Delta(f_2,f_{\min})=22.90\%$. We also have $\x^{*}=( 0.6252,
  0.4015,
 -0.5397,
 -0.1415,
  0.3697)$.
\par
\noindent\textbf{Performance of $\HH$.} Let us denote by $\widetilde{f}_{d,d_c,k}$ the output of $\HH$ at iteration $k\geq1$, where $d_c\in\set{1,\dots,d}$ is the degree of Christoffel polynomials used for modifying the feasible set. Then, there are two possible ways in which our iterative approach can be coupled with high-order relaxations.
\begin{itemize}
    \item  Firstly, one can modify the feasible set using the Christoffel polynomials constructed from pseudo-moment matrices of order $d$. As the higher-order pseudo-moment matrices contain information about monomials of higher degrees, the resulting sublevel sets, possible nonconvex, are expected to be much more precise, a phenomenon that has already been described in Figure \ref{fig: supportidentification}.  Consequently, the number of iterations to reach satisfactory bound improvements should intuitively be smaller. For the optimization problem in \eqref{unionballs}, we obtain $\widetilde{f}_{2,2,5}=-6.7029$, yielding $\Delta(\widetilde{f}_{2,2,5},f_{\min})=14.72\%$.
    \item Secondly, from the pseudo-moment matrix of order $d$, one can extract a submatrix of order $d_c < d $, and modify the feasible set using less complicated level sets of low order Christoffel polynomials. This can be beneficial, for example, in high-dimensional settings, where construction of high-degree Christoffel polynomials can be challenging. In the particular case of \eqref{unionballs}, we get $\widetilde{f}_{2,1,5}=-6.6961$, so that $\Delta(\widetilde{f}_{2,1,5},f_{\min})=14.64\%$.
\end{itemize}
Our numerical experiments highlight the significant role of polynomials in the kernel of pseudo-moment matrices for $d \geq 2$. This should be carefully considered when $d_c = d$, as it requires an appropriate adjustment of $\beta$. For instance, second-order relaxation generally appears less accurate for problems with nonconvex and/or disconnected feasible sets, such as \eqref{unionballs}.
As a result, even second-order Christoffel polynomials may fail to accurately capture the true location of the minimizer, particularly when $\beta$ is set very close to zero.
In contrast, selecting $d_c = 1$, appears to be a more conservative approach for reducing the size of the feasible set, typically yielding slower improvements in the bounds. Hence, without prior knowledge of higher-order relaxation quality, applying the bound-strengthening technique with $d_c = 1$ is a safer and effective approach.
\par
\textbf{Performance of $\hh$.} We remind that $\hh$ is always applied using the marginal Christoffel polynomials with $d_c=1$. In case of higher-order relaxations, we expect the pseudo-moment matrices to contain more information, which, in turn, can increase the efficiency of $\hh$. The available local solution is $\overline{\x}=(1.2602, 0.9712, 0.9292, 0.8395, 1.0262)$.\\
After constructing the five marginal Christoffel polynomials, we obtain the following list of candidate thresholds (see \eqref{eq: lambda_i}) :
$$
(\lambda_1,\lambda_2,\lambda_3,\lambda_4,\lambda_5)=(1.2059, 2.5729, 4.2559, 8.3069, 1.5804).
$$
Hence, by setting $\tau=1.5$, the algorithm $\hh$ applies the feasible set reduction with respect to the first coordinate only, yielding a new bound $\widetilde{f}_{2,1}=-6.1883$, and a new optimality gap $\Delta(\widetilde{f}_{2,1},f_{\min})=7.63\%$.
Using smaller values of the threshold filtering parameter $\tau$ leaves the initial relaxation bound $f_2$ unchanged.

\section{Conclusion}
\label{sec:conclusion}
\noindent
To mitigate the computational complexity of high-order SOS-relaxations, our paper shows that it is possible to enhance the accuracy of low-order relaxations by leveraging Christoffel-Darboux kernels. We propose two heuristics that use information available at order $d$ to construct further polynomial restrictions without increasing the relaxation order.
\par\noindent
The iterative method $\HH$ excludes successive optimal pseudo-moment solutions, while $\hh$ combines marginal Christoffel scores with an available local solution to select coordinatewise restrictions.
Our experimental results demonstrate the advantages of these two approaches, particularly for various types of QCQPs, where they often yield significantly improved lower bounds up to $1000\times$ faster than solving the computationally expensive higher-order relaxations.
We also introduced $\HHCS$ and $\hhcs$, which are refinements of $\HH$ and $\hh$, designed to exploit correlative sparsity. Despite the inherently sparse structure of many POP instances, ascending the hierarchy remains challenging. In our experiments, the sparse variants also improve the computed relaxation values without moving to higher orders.
The theoretical analysis complements these numerical findings. For $\HH$, uniform separation quantifies the effect of successive restrictions on the SDP solutions; the exact singular knapsack example further exhibits geometric convergence at fixed order. For $\hh$, the Boolean analysis links marginal Christoffel scores to the pseudo-moment behavior of dominated coordinates. These results clarify particular features of both mechanisms, while the added constraints remain heuristic on general instances.
\par
Tight lower bounds are crucial in many practical applications, motivating the need for broader exploration of bound-strengthening techniques like ours. For example, integrating our heuristic bounds into a general-purpose Branch \& Bound solver could reduce solving times by limiting the number of branchings. A natural extension would be to develop an adaptive scheme for selecting the penalization threshold in $\HH$ and $\HHCS$, applying milder penalties when a rank reduction in the moment matrix is detected. Extending these methods to more general POPs, including those with discrete constraints, is another promising direction for future work.
\hfill\break\\\noindent
\textbf{Acknowledgements}.
\noindent
This work was supported by the European Union’s HORIZON–MSCA-2023-DN-JD programme under the Horizon Europe (HORIZON) Marie Skłodowska-Curie Actions, grant agreement 101120296 (TENORS),
the AI Interdisciplinary Institute ANITI funding, through the French ``Investing for the Future PIA3'' program under the Grant agreement n° ANR-19-PI3A-0004 as well as by
the National Research Foundation, Prime Minister’s Office, Singapore under its Campus for Research Excellence and Technological Enterprise (CREATE) programme.
This work was partially performed using HPC resources from CALMIP (Grant 2023-P23035).

\phantomsection\label{sec:appendix}\section*{Appendix}
\noindent

\subsection*{A.1. Proof of Theorem~\ref{thm:H1-separation}}
\phantomsection\label{app:H1}

\noindent
Fix an iteration $k$, and let $e_i$ and $\p_i$ denote the eigenvalues
and orthonormal eigenvectors of ${\bf M}_d(\y_k^*)$, respectively.
The associated eigenpolynomials are
$p_i(\x)=\p_i^\top\vb_d(\x)$.
Recall that, if $r_k$ is the kernel dimension, then
\[
\gamma_k=\sum_{i=1}^{s(n,d)-r_k}\frac{e_i}{e_i+\beta}.
\]

\noindent
At iteration $k$, $\HH$ introduces the polynomial constraint
\[
\sum_{i=1}^{s(n,d)-r_k}\frac{p_i^2(\x)}{e_i+\beta}
\le(1-\varepsilon)\gamma_k.
\]
All previously introduced constraints are retained by $\HH$.
Taking the entry indexed by the constant monomial in the localizing
constraint therefore gives, for every $j>k$,
\[
L_{\y_k^*}\left(
\sum_{i=1}^{s(n,d)-r_k}\frac{p_i^2}{e_i+\beta}
\right)=\gamma_k, \quad \text{and}\quad L_{\y_j^*}\left(
\sum_{i=1}^{s(n,d)-r_k}\frac{p_i^2}{e_i+\beta}
\right)
\le(1-\varepsilon)\gamma_k.
\]
Subtracting these expressions gives
\begin{align}\label{eq:H1-separation-proof1}
\varepsilon\gamma_k
\le
\sum_{i=1}^{s(n,d)-r_k}
\frac{L_{\y_k^*}(p_i^2)-L_{\y_j^*}(p_i^2)}
{e_i+\beta}&=\operatorname{Tr}\left[
\left(
\sum_{i=1}^{s(n,d)-r_k}
\frac{\p_i\p_i^\top}{e_i+\beta}
\right)
\left(
{\bf M}_d(\y_k^*)-{\bf M}_d(\y_j^*)
\right)
\right]\\
&\leq \left\|
\sum_{i=1}^{s(n,d)-r_k}
\frac{\p_i\p_i^\top}{e_i+\beta}
\right\|_F
\left\|
{\bf M}_d(\y_k^*)-{\bf M}_d(\y_j^*)
\right\|_F,\label{eq:H1-separation-proof2}
\end{align}
where the last step follows from the Cauchy--Schwarz inequality for the Frobenius
inner product.
\par\noindent
Since the vectors $\p_i$ are orthonormal and $\beta>0$,
\[
\left\|
\sum_{i=1}^{s(n,d)-r_k}
\frac{\p_i\p_i^\top}{e_i+\beta}
\right\|_F^2
=
\sum_{i=1}^{s(n,d)-r_k}\frac{1}{(e_i+\beta)^2}
\le\frac{s(n,d)}{\beta^2}.
\]
Furthermore, the normalization $y_{k,\mathbf{0}}^*=1$
implies that the largest eigenvalue $e_1$ of
${\bf M}_d(\y_k^*)$ satisfies $e_1\ge1$. Therefore,
\[
\gamma_k
=\sum_{i=1}^{s(n,d)-r_k}\frac{e_i}{e_i+\beta}
\ge\frac{e_1}{e_1+\beta}
\ge\frac{1}{1+\beta}.
\]
Substituting these two estimates into
\eqref{eq:H1-separation-proof2}, we obtain
\[
\left\|
{\bf M}_d(\y_j^*)-{\bf M}_d(\y_k^*)
\right\|_F
\ge
\frac{\varepsilon\gamma_k\beta}{\sqrt{s(n,d)}}
\ge
\frac{\varepsilon\beta}
{(1+\beta)\sqrt{s(n,d)}},
\]
which proves the result.

\subsection*{A.2. A one-split $0$--$1$ knapsack family}
\phantomsection\label{app:H1-class}

\noindent
We consider a class of $0$--$1$ knapsack problems for which the
effect of successive $\HH$ restrictions can be described explicitly.
For this example, we use the exact singular construction
$\beta=0$, with the kernel conditions imposed as polynomial
equalities.\par\noindent
Consider
\[
\max\left\{
\sum_{i=1}^n\pi_i x_i:
\sum_{i=1}^n a_i x_i\le b,\quad x_i^2=x_i
\right\},
\]
where $a_i,\pi_i>0$ and
\[
\frac{\pi_1}{a_1}>\cdots>\frac{\pi_n}{a_n}.
\]
The continuous knapsack relaxation, obtained by replacing
$x_i^2=x_i$ with $0\le x_i\le1$, is solved by taking items
in decreasing order of their profit-to-weight ratios.

\par\noindent
Suppose that its optimal solution takes all items
$1,\ldots,j-1$, a fraction
\[
\theta_0:=
\frac{b-\sum_{i<j}a_i}{a_j}\in(0,1)
\]
of item $j$, and none of the remaining items.
Let $\x^{\mathrm g}$ denote the feasible binary point taking
only items $1,\ldots,j-1$. The fractional solution can
therefore be written as
\[
\x^{\mathrm g}+\theta_0\bm e_j,
\]
where $\bm e_j$ is the $j$th canonical basis vector.

\par\noindent
At order one, the moment relaxation has the same optimal
first moments as the continuous knapsack relaxation.
Indeed, the constraints $x_i^2=x_i$ imply
$0\le L_{\y}(x_i)\le1$, while the capacity constraint
remains linear in these first moments. Conversely, independent Bernoulli variables with any prescribed
first moments in $[0,1]^n$ provide a positive semidefinite
moment matrix, so every point of the continuous relaxation
admits such a completion.

\par\noindent
At this optimum, the Boolean relations and positive
semidefiniteness force the rows corresponding to the integral
coordinates to be multiples of the constant row. Only the
coordinate $j$ can have nonzero variance. Thus the entire
optimal moment matrix is
\[
{\bf M}_1(\y_0^*)
=(1-\theta_0)
\vb_1(\x^{\mathrm g})\vb_1(\x^{\mathrm g})^\top
+\theta_0
\vb_1(\x^{\mathrm g}+\bm e_j)
\vb_1(\x^{\mathrm g}+\bm e_j)^\top.
\]
Thus, the pseudo-moments have a two-point representation,
with weights $1-\theta_0$ and $\theta_0$. Notice that
$\x^{\mathrm g}+\bm e_j$ violates the knapsack capacity;
this representation is therefore not a probability measure
supported on the original feasible set.

\par\noindent
We next identify conditions under which $\x^{\mathrm g}$
is globally optimal. Suppose that all profits are integers
and $\theta_0\pi_j<1$. The continuous relaxation gives
the upper bound
\[
\sum_{i<j}\pi_i+\theta_0\pi_j
<
\sum_{i<j}\pi_i+1.
\]
Since the profit of every feasible binary point is an
integer, no such point can have a profit greater than
$\sum_{i<j}\pi_i$. Hence $\x^{\mathrm g}$ is a global
minimizer of the equivalent problem with objective
$f(\x)=-\sum_i\pi_i x_i$.

\par\noindent
We now examine the effect of $\HH$. The initial moment
matrix has rank two, and its kernel equations restrict
the subsequent order-one pseudo-moment matrices to the
same two-dimensional range. Consequently, every later
feasible pseudo-moment matrix has the form
\[
{\bf M}_1(\y)
=(1-\theta)
\vb_1(\x^{\mathrm g})\vb_1(\x^{\mathrm g})^\top
+\theta
\vb_1(\x^{\mathrm g}+\bm e_j)
\vb_1(\x^{\mathrm g}+\bm e_j)^\top,
\]
for some $\theta\in[0,\theta_0]$.

\par\noindent
Suppose that the optimal pseudo-moment matrix at iteration
$k$ corresponds to $\theta_k>0$. In this singular setting,
we use the positive-spectrum Christoffel polynomial
$\Lambda_1^{\y_k^*}(\x):=\vb_1(\x)^\top
{\bf M}_1(\y_k^*)^\dagger\vb_1(\x)$,
where $\dagger$ denotes the Moore--Penrose pseudoinverse.
Its values at the two endpoints are
\[
\Lambda_1^{\y_k^*}(\x^{\mathrm g})
=\frac{1}{1-\theta_k},
\qquad
\Lambda_1^{\y_k^*}(\x^{\mathrm g}+\bm e_j)
=\frac{1}{\theta_k}.
\]
Since $\gamma_k=2$, the next $\HH$ restriction imposes
\[
\frac{1-\theta}{1-\theta_k}
+\frac{\theta}{\theta_k}
\le 2(1-\varepsilon).
\]

\par\noindent
We first determine when this restriction preserves the
global minimizer $\x^{\mathrm g}$, which corresponds to
$\theta=0$. Substituting this value gives
\begin{align}\label{eq:knapsack-safe-app}
    \frac{1}{1-\theta_k}\le2(1-\varepsilon)\iff \theta_k\le
\frac{1-2\varepsilon}{2(1-\varepsilon)}.
\end{align}
In particular, this condition requires
$0<\varepsilon<1/2$ and implies $\theta_k<1/2$.
\par\noindent
Under~\eqref{eq:knapsack-safe-app}, the left-hand side
of the Christoffel inequality is increasing in $\theta$.
Since the objective favors larger values of $\theta$,
the next relaxation attains the largest admissible value,
namely
\[
\theta_{k+1}
=
\theta_k
\frac{1-2\varepsilon-2(1-\varepsilon)\theta_k}
{1-2\theta_k}.
\]
Moreover,
\[
(1-2\varepsilon)\theta_k-\theta_{k+1}
=
\frac{2\varepsilon\theta_k^2}{1-2\theta_k}
\ge0.
\]
It follows that
\[
0\le\theta_{k+1}
\le(1-2\varepsilon)\theta_k.
\]
Therefore, if~\eqref{eq:knapsack-safe-app} holds initially,
it continues to hold at every iteration. In particular,
$\x^{\mathrm g}$ remains feasible for all the generated
restrictions, while
\[
\theta_k\le(1-2\varepsilon)^k\theta_0.
\]

\par\noindent
Finally, since $\x^{\mathrm g}$ is globally optimal,
the relaxation bound at iteration $k$ satisfies
\[
\widetilde f_{1,k}
=f_{\min}-\pi_j\theta_k.
\]
Consequently,
\[
0\le f_{\min}-\widetilde f_{1,k}
\le\pi_j(1-2\varepsilon)^k\theta_0.
\]
Thus, for this knapsack subclass, the exact singular
version of $\HH$ preserves a global minimizer and its
first-order relaxation bounds converge geometrically
to $f_{\min}$.\\
Therefore, for this specific problem class, we can illustrate that successive Christoffel restrictions reduce the relaxation gap geometrically at the fixed order \(d=1\), while preserving the true optimum.

\subsection*{A.3. Boolean Ising/QUBO problems}
\phantomsection\label{app:H2-persist-new}

\noindent
\textbf{Proof of Proposition~\ref{prop:H2-persist-new}.}
Suppose that $|h_i|>S_i$. For any choice of the other
variables, the terms of the objective involving $x_i$ are
\[
x_i\left(h_i+\sum_{j\ne i}J_{ij}x_j\right).
\]
Since
\[
\left|\sum_{j\ne i}J_{ij}x_j\right|
\le S_i<|h_i|,
\]
the expression in parentheses always has the sign of $h_i$.
Consequently, every global minimizer satisfies
$x_i=-\operatorname{sign}(h_i)$.

\par\noindent
We next establish the corresponding property of the
first-order moment relaxation. Let $\y^*$ be any optimal
pseudo-moment sequence. Its moment matrix
${\bf M}_1(\y^*)$, indexed by $1,x_1,\ldots,x_n$,
is positive semidefinite. Moreover, the Boolean equations
$x_i^2=1$ imply that all its diagonal entries equal one.
Thus, there exist unit vectors $u_0,u_1,\ldots,u_n$
such that
\[
m_i=L_{\y^*}(x_i)=\langle u_0,u_i\rangle,
\qquad
L_{\y^*}(x_ix_j)=\langle u_i,u_j\rangle.
\]
In this representation, the terms of the SDP objective
involving $u_i$ are
\[
\langle u_i,w_i\rangle,
\qquad
w_i:=h_i u_0+\sum_{j\ne i}J_{ij}u_j.
\]
By the triangle inequality,
\[
\|w_i\|\ge |h_i|-S_i>0.
\]
Fix $u_0$ and $u_j$ for $j\ne i$. Any unit vector $u_i$
still defines a feasible first-order moment matrix through
its Gram representation. Since the SDP solution is optimal,
$u_i$ must minimize $\langle u_i,w_i\rangle$, and hence
\[
u_i=-\frac{w_i}{\|w_i\|}.
\]
Taking the inner product with $u_0$ gives
\[
m_i
=-\frac{h_i+\sum_{j\ne i}J_{ij}m_j}{\|w_i\|}.
\]
Since $|m_j|\le1$, we have
\[
\left|\sum_{j\ne i}J_{ij}m_j\right|\le S_i,
\qquad
\|w_i\|\le |h_i|+S_i.
\]
The numerator consequently has the sign of $h_i$,
and we obtain
\[
\operatorname{sign}(m_i)=-\operatorname{sign}(h_i),
\qquad
|m_i|\ge\frac{|h_i|-S_i}{|h_i|+S_i}.
\]

\par\noindent
We now relate these properties to the marginal constraints
introduced by $\hh$. By~\eqref{eq:H2-Boolean-rule-new},
whenever $\gamma_i\le\tau$, the added constraint imposes
$x_i=\operatorname{sign}(m_i)$ on the Boolean domain.
As this equals $-\operatorname{sign}(h_i)$, the constraint
preserves every global minimizer.

\par\noindent
Finally, suppose that
$\overline{x}_i=-\operatorname{sign}(h_i)$ and
that~\eqref{eq:H2-R-new} holds. The preceding bound yields
\[
|m_i|
\ge\frac{|h_i|-S_i}{|h_i|+S_i}
\ge c_\tau.
\]
Together with $\overline{x}_i=\operatorname{sign}(m_i)$,
this implies $\gamma_i\le\tau$ by
\eqref{eq:H2-Boolean-rule-new}. Hence $\hh$ adds the
corresponding marginal constraint, completing the proof.
\noindent
\subsection*{A.4. Unrestricted iH2: implementation and numerical results}
\phantomsection\label{app:iH2}
The iterative version $\hhi$ of our local-solution-based heuristic is described in Algorithm \ref{alg:iteratie_loc-sol-based-heuristic}.
\begin{algorithm}[H]
\footnotesize
\caption{Implementing $\hhi$}
\label{alg:iteratie_loc-sol-based-heuristic}
\begin{algorithmic}[1]
\Statex \textbf{Input:} relaxation order $d$, tolerance $\delta>0$, maximum number of iterations $N$, filtering parameter $\tau>1$.
\State $k\gets0$, $\widetilde K_0\gets K$; solve the order-$d$ relaxation and recover $(\y_0^*,\widetilde f_{d,0})$, a local solution $\bar\x_0$, $\operatorname{ub}_0=f(\bar\x_0)$, and $\Delta_0=\Delta(\operatorname{ub}_0,\widetilde f_{d,0})$.
\While{$k<N$ and $\widetilde f_{d,k}\le\operatorname{ub}_k$ and $\Delta_k>\delta$}
\State For every $i$, construct $\widetilde\Lambda_1^{\y^*_{k[i]}}$ and set $\gamma_{i,k}=\widetilde\Lambda_1^{\y^*_{k[i]}}(\bar x_{k,i})$.
\State $\displaystyle \widetilde K_{k+1}\gets\widetilde K_k\cap\bigcap_{i:\gamma_{i,k}\le\tau}\{\x:\ x_i\in\widetilde S_1(\y^*_{k[i]},\gamma_{i,k})\}$.
\State Solve the order-$d$ relaxation over $\widetilde K_{k+1}$ and recover $\widetilde f_{d,k+1}$, $\y^*_{k+1}$, and a new local solution $\bar\x_{k+1}$.
\State $\operatorname{ub}_{k+1}\gets\min\{\operatorname{ub}_k,f(\bar\x_{k+1})\}$, $\Delta_{k+1}\gets\Delta(\operatorname{ub}_{k+1},\widetilde f_{d,k+1})$, $k\gets k+1$.
\EndWhile
\Statex \textbf{Output:} $(\widetilde f_{d,k})_{k\ge1}$ with $\widetilde f_{d,k+1}\ge\widetilde f_{d,k}\ge f_d$.
\end{algorithmic}
\end{algorithm}

\noindent
Moreover, $\hhi$ was applied to the QCQP instances from Section~\ref{sec:num_exp}, with up to $N = 15$ iterations. The results, summarized in Table \ref{tab:IterativeH2-Qcombined-revised}, show that the iterative variant is even more effective in reducing gaps or solving instances, especially when the filtering parameter is moderate, e.g. $\tau=1.25$. Notice that $\hhi$ requires, on average, less iterations than the alternative iterative approach $\HH$.
\begin{table}[H]
\centering \scriptsize
\renewcommand{\arraystretch}{1.08}
\begin{tabular}{|c|c|c|c|c|c|c|c|c|c|}
    \hline
    \multirow{2}{*}{$n$} & \multirow{2}{*}{$s$} & \multirow{2}{*}{$\tau$} & \multirow{2}{*}{$\widetilde{f}_1>f_{\min}$} & \multicolumn{2}{c|}{Gap (\%)} & \multicolumn{2}{c|}{Time ($\mathrm{s}$)} & \multirow{2}{*}{Solved} & \multirow{2}{*}{$k$} \\
    \cline{5-6} \cline{7-8}
    & & & & $\Delta(f_1,f_{\min})$ & $\Delta(\widetilde{f}_1,f_{\min})$ & $\widetilde{f}_1$ & $f_2$ & & \\
    \hline
    \addlinespace
    \hline
    \multirow{8}{*}{\rotatebox{90}{$20$}}
        & \multirow{4}{*}{\rotatebox{90}{$0\:\%$}}
        & None  & \cellcolor{gray!15}6 & \cellcolor{gray!15}4.381 & \cellcolor{gray!15}0.121 & \cellcolor{gray!15}0.069 & \cellcolor{gray!15}26.489 & \cellcolor{gray!15}42 & \cellcolor{gray!15}1.16 \\
    \cline{3-10}
         &  & $1.5$  & \cellcolor{gray!15}8 & \cellcolor{gray!15}4.076 & \cellcolor{gray!15}0.092 & \cellcolor{gray!15}0.074 & \cellcolor{gray!15}26.476 & \cellcolor{gray!15}42 & \cellcolor{gray!15}1.19 \\
   \cline{3-10}
        &  & $1.25$ & \cellcolor{gray!15}4 & \cellcolor{gray!15}4.490 & \cellcolor{gray!15}0.183 & \cellcolor{gray!15}0.110 & \cellcolor{gray!15}26.443 & \cellcolor{gray!15}45 & \cellcolor{gray!15}1.91 \\
    \cline{3-10}
        &  & $1.1$  & \cellcolor{gray!15}3 & \cellcolor{gray!15}4.739 & \cellcolor{gray!15}1.457 & \cellcolor{gray!15}0.461 & \cellcolor{gray!15}26.396 & \cellcolor{gray!15}31 & \cellcolor{gray!15}5.95 \\
    \cline{2-10}
        & \multirow{4}{*}{\rotatebox{90}{$20\:\%$}}
        & None  & 9 & 4.228 & 0.179 & 0.064 & 26.689 & 40 & 1.12 \\
    \cline{3-10}
        &  & $1.5$  & 9 & 4.279 & 0.217 & 0.090 & 26.628 & 39 & 1.53 \\
    \cline{3-10}
        &  & $1.25$ & 7 & 4.342 & 0.446 & 0.146 & 26.784 & 38 & 2.58 \\
    \cline{3-10}
        &  & $1.1$  & 2 & 4.681 & 1.038 & 0.343 & 26.655 & 32 & 6.02 \\
    \hline
    \addlinespace
    \hline
    \multirow{8}{*}{\rotatebox{90}{$30$}}
        & \multirow{4}{*}{\rotatebox{90}{$0\:\%$}}
        & None  & \cellcolor{gray!15}12 & \cellcolor{gray!15}4.880 & \cellcolor{gray!15}0.122 & \cellcolor{gray!15}0.304 & \cellcolor{gray!15}1209.058 & \cellcolor{gray!15}38 & \cellcolor{gray!15}1.19 \\
    \cline{3-10}
       &  & $1.5$  & \cellcolor{gray!15}11 & \cellcolor{gray!15}4.822 & \cellcolor{gray!15}0.109 & \cellcolor{gray!15}0.417 & \cellcolor{gray!15}1210.032 & \cellcolor{gray!15}38 & \cellcolor{gray!15}1.73 \\
    \cline{3-10}
        &  & $1.25$ & \cellcolor{gray!15}9 & \cellcolor{gray!15}5.166 & \cellcolor{gray!15}0.213 & \cellcolor{gray!15}0.690 & \cellcolor{gray!15}1204.214 & \cellcolor{gray!15}38 & \cellcolor{gray!15}2.87 \\
    \cline{3-10}
        &  & $1.1$  & \cellcolor{gray!15}4 & \cellcolor{gray!15}5.350 & \cellcolor{gray!15}1.561 & \cellcolor{gray!15}1.796 & \cellcolor{gray!15}1205.919 & \cellcolor{gray!15}24 & \cellcolor{gray!15}7.77 \\
    \cline{2-10}
        & \multirow{4}{*}{\rotatebox{90}{$20\:\%$}}
        & None  & 13 & 5.237 & 0.253 & 0.327 & 1184.165 & 33 & 1.24 \\
    \cline{3-10}
     &  & $1.5$  & 13 & 5.123 & 0.192 & 0.335 & 1181.814 & 35 & 1.32 \\
    \cline{3-10}
        &  & $1.25$ & 5 & 5.553 & 0.258 & 0.651 & 1170.509 & 41 & 2.68 \\
    \cline{3-10}
        &  & $1.1$  & 0 & 6.058 & 1.452 & 2.111 & 1163.133 & 22 & 9.28 \\
    \hline
\end{tabular}
\caption{Performance of $\hhi$ in dimensions $20$ and $30$, for different values of the filtering parameter $\tau$ (``None'' corresponds to cases where no filtering is used).}
\label{tab:IterativeH2-Qcombined-revised}
\end{table}

\noindent
\subsection*{A.5. Sparsity adaptations}
\phantomsection\label{app:sparsity}
\noindent
The correlative-sparsity variants $\HHCS$ and $\hhcs$ adapt
Algorithms~\ref{alg:iteration-based-heuristic} and~\ref{alg:loc-sol-based-heuristic}
by constructing Christoffel information clique-wise from the sparse
pseudo-moment sequences. Their implementations are given below.
\begin{algorithm}[H]
\footnotesize
\caption{Implementing $\hhcs$}
\label{alg:loc-sol-based-heuristic - CS}
\begin{algorithmic}[1]
\Statex \textbf{Input:} relaxation order $d$, filtering parameter $\tau>1$.
\State Solve the sparse order-$d$ relaxation; recover $(\y_l^*)_{l=1}^p$ and a local solution $\bar\x$.
\State For each clique $I_l$ and $i\in I_l$, construct $\widetilde\Lambda_1^{\y^*_{l,[i]}}$ and set $\gamma_i=\widetilde\Lambda_1^{\y^*_{l,[i]}}(\bar x_i)$.
\State Set $\displaystyle \widetilde K\gets K\cap\bigcap_{l=1}^p\bigcap_{i\in I_l}\{\x:x_i\in\widetilde S_1(\y^*_{l,[i]},\gamma_i)\}$ and re-solve the sparse order-$d$ relaxation.
\Statex \textbf{Output:} $\widetilde f_d^{\cs}\ge f_d^{\cs}$.
\end{algorithmic}
\end{algorithm}
\vspace{-2mm}
\begin{algorithm}[H]
\footnotesize
\caption{Implementing $\HHCS$}
\label{alg:iteration-based-heuristic_cd_modified}
\begin{algorithmic}[1]
\Statex \textbf{Input:} $d$, $\delta>0$, $N$, $\varepsilon\in(0,1)$.
\State $k\gets0$, $\widetilde K_0\gets K$; solve the sparse order-$d$ relaxation and recover $(\y^*_{l,0})_{l=1}^p$, $\widetilde f^{\cs}_{d,0}=f_d^{\cs}$, a local solution $\bar\x_0$, and the initial gap $\Delta_0$.
\While{$k<N$, $\widetilde f^{\cs}_{d,k}\le\operatorname{ub}_k$, and $\Delta_k>\delta$}
\State For each clique $l$, construct $\widetilde\Lambda_d^{\y^*_{l,k}}$ and $\gamma_{l,k}=L_{\y^*_{l,k}}(\widetilde\Lambda_d^{\y^*_{l,k}})$.
\State $\displaystyle \widetilde K_{k+1}\gets\widetilde K_k\cap\bigcap_{l=1}^p\widetilde S_d(\y^*_{l,k},(1-\varepsilon)\gamma_{l,k})$.
\State Re-solve the sparse order-$d$ relaxation; update $\widetilde f^{\cs}_{d,k+1}$, $(\y^*_{l,k+1})_{l=1}^p$, $\bar\x_{k+1}$, $\operatorname{ub}_{k+1}$, and $\Delta_{k+1}$; set $k\gets k+1$.
\EndWhile
\Statex \textbf{Output:} $(\widetilde f^{\cs}_{d,k})_{k\ge1}$ with $\widetilde f^{\cs}_{d,k+1}\ge\widetilde f^{\cs}_{d,k}\ge f_d^{\cs}$.
\end{algorithmic}
\end{algorithm}
\FloatBarrier
\bibliographystyle{siam}
\bibliography{MomentSOSbib}
\end{document}